\documentclass[12pt]{amsart}
\usepackage{amsfonts}
\usepackage{amsmath}
\usepackage{amsxtra}
\usepackage{amssymb,latexsym}
\usepackage[mathcal]{eucal}
\usepackage{amscd}

\makeindex

\newtheorem{theo}{{\bfseries Theorem}}[section]
\newtheorem{prop}[theo]{{\bfseries Proposition}}
\newtheorem{lem}[theo]{{\bfseries Lemma}}
\newtheorem{cor}[theo]{{\bfseries Corollary}}
\newtheorem{df}[theo]{{\bfseries Definition}}

\newcommand \ol {\overline}
\newcommand \N {\mathbb N}

\newcommand \Z {\mathbb Z}

\newcommand \A {\mathcal A}

\newcommand \CC {\mathcal C}

\newcommand \I {\mathcal I}

\newcommand \NN {\mathcal N}
\newcommand \OO {\mathcal O}

\newcommand \QQ {\mathcal Q}

\newcommand \al {\alpha }
\newcommand \bt {\beta}
\newcommand \ep {\epsilon}
\newcommand \dl {\delta}

\newcommand \tto {\twoheadrightarrow}

\usepackage{amssymb,latexsym}
\usepackage[mathcal]{eucal}
\usepackage{amscd}
\usepackage{amsmath}
\usepackage{graphicx}
\usepackage{amsfonts}
\usepackage{amscd}

\numberwithin{equation}{section}

\begin{document}

\title[Cantor Set Automorphism Group]{ Conjugacy in the Cantor Set Automorphism Group}
\vspace{1cm}

\author{Ethan Akin}
\address{Mathematics Department,
 The City College, 137 Street and Convent Avenue,
 New York City, NY 10031, USA}

\email{ethanakin@earthlink.net}
%


    \vspace{.5cm}
\date{February, 2015}

\begin{abstract} We survey, and extend, results on the adjoint action of the homeomorphism group $H(X)$ on
the space of surjective continuous maps, $C_s(X)$, where $X$ is a Cantor set. We look also at the restriction
of the action to various dynamically defined subsets of $C_s(X)$, e. g. the sets of topologically transitive maps,
chain transitive maps, chain mixing maps, etc.  In each case, we consider whether there exist elements with
a dense conjugacy class and if so, what the generic elements look like.
\end{abstract}
\vspace{.5cm}

\maketitle

\tableofcontents

\section*{Introduction}\label{intro}
\vspace{.5cm}

When a group acts on itself by the adjoint action, the orbit of an element is its conjugacy class.  In the case of the group
$H(X)$\index{$H(X)$} of homeomorphisms on a compact metric space $X$, the adjoint action extends to $C(X)$,\index{$C(X)$} the space of all continuous maps
on $X$. For a continuous map $f$ on $X$ we refer to its $H(X)$ orbit as its conjugacy class and denote it by $H(X) \cdot f$.

The case of $H(X)$ is of special interest because of its dynamic interpretation.  We can
regard a map $f \in C(X)$ as a discrete time
dynamical system on
the state space $X$, describing the evolution by iteration: $x_{n+1} = f(x_n)$. The maps $f$ and $g$ are conjugate
precisely when there exists $h \in H(X)$ such that $g = h \circ f \circ h^{-1}$. This says that $h$ is an isomorphism from
$f$ to $g$ in the category of dynamical systems.

It is convenient to restrict attention to surjective maps, which form a closed subset $C_s(X) \subset C(X)$\index{$C_s(X)$}. A dynamical system
is a pair $(X,f)$
with $X$ a nonempty, compact metric space and $f \in C_s(X)$. It is invertible when $f$ is injective, so that $f \in H(X)$.

Considerable work has been done analyzing the adjoint action of $H(X)$ for the special case when $X$ is a Cantor space,
i.e. a metric space homeomorphic to the usual Cantor set.
The basis of clopen
sets makes everything more tractable in that case. Furthermore, the Cantor set plays a central role in the theory of dynamical
systems.  In various guises it appears as the state space, e.g. for
coding-related systems like subshifts and for algebraic systems like
adding machines (also called odometers).  Furthermore, every system on a  state space with no isolated points
has an almost one-to-one lift to a system on
the Cantor set.

For a Polish topological group $G$ like $H(X)$, it is of special
interest when the adjoint action of $G$ on itself is topologically
transitive, i.e. when there exists $g$ whose conjugacy class is dense in $G$. Such a $g \in G$ is called
a \emph{transitive element}\index{transitive element}. A group which admits such transitive elements is said to have the \emph{Rohlin Property}.
\index{Rohlin Property}
The name is motivated by an ergodic theory result of Rohlin concerning the automorphism group of the Lebesgue space. When $G$
satisfies the Rohlin Property then the set of transitive elements is a dense $G_{\dl}$ subset of $G$. It sometimes happens
that a single conjugacy class is residual, i.e. it contains a dense $G_{\dl}$ set. Since distinct conjugacy classes are disjoint
but any two residual subsets meet, it follows that there is at most one residual conjugacy class. When such a class exists
we call its members \emph{transitive elements of residual type}\index{transitive element of residual type} and we say that the 
group has the \emph{Strong Rohlin Property}.\index{Strong Rohlin Property}
\index{Rohlin Property!Strong}
Finally, when the diagonal action of $G$ on any finite product $G^n$ admits a residual conjugacy class then we say that
$G$ has \emph{ample generics}. For a survey, see \cite{GW2}.

For $X$ a Cantor space it was shown by Glasner and Weiss in \cite{GW} and,
independently, in \cite{AHK} that the automorphism group $H(X)$
has the Rohlin Property. For a certain class of good measures $\mu$ on $X$ it was shown in \cite{A2} that $H_{\mu}(X)$,\index{$H_{\mu}(X)$}
the closed subgroup of automorphisms which preserve $\mu$, has the Strong Rohlin Property.  It was shown that the product
of the universal adding machine (see below) with the identity on a Cantor space provides a transitive element of residual type.
Using Fra\"{i}ss\'{e}
theory, Kechris and Rosendal showed in \cite{KR} that $H(X)$ has the Strong Rohlin Property. An explicit description
of a transitive element of residual type was then given in \cite{AGW}. With $\mu$ the Haar measure on $X = 2^{\N}$ it was shown
in \cite{KR} that $H_{\mu}(X)$ has ample generics and in \cite{K} that $H(X)$ itself has ample generics. Transitivity results with
different topologies were proved in \cite{BDK}.

A conjugacy invariant subset of $C(X)$ defines a dynamic property, i.e. one which is an invariant under topological conjugacy.
Conversely, we will consider the subsets defined by  various dynamic concepts
associated with recurrence, transitivity
and mixing. We recall that the following subsets of $C_s(X)$ are closed.
\begin{itemize}
\item  $C_s(X;1)$,\index{$C_s(X;1)$} the surjective maps which admit fixed points.
\item  $CT(X)$,\index{$CT(X)$} the chain transitive maps, which contains $TT(X)$, the topologically transitive maps, as a dense $G_{\dl}$ subset.
\item  $CR(X)$,\index{$CR(X)$} the chain recurrent maps, which contains the maps with dense recurrent points as a dense $G_{\dl}$ subset.
\item $CM(X)$,\index{$CM(X)$} the chain mixing maps, which contains $ WM(X)$, the weak mixing maps, as a dense $G_{\dl}$ subset.
\end{itemize}
So $CR(X;1) \ = \ CR(X) \cap C_s(X;1)$ and $CM(X;1) \ = \ CM(X) \cap C_s(X;1)$ \index{$CR(X;1)$}\index{$CM(X;1)$} are closed as well. If a chain
transitive map admits a
fixed point then it is chain mixing and so we need not consider $CT(X;1)$. Note that the density results assume that $X$ is a Cantor
space.

For such a conjugacy invariant subset $S$  of $C_s(X)$ with $X$ a Cantor space,
 we wish to consider whether the restriction of the adjoint action to $S$ admits
 transitive elements or transitive elements of residual type.  For
example, Hochman \cite{H} proved that the universal adding machine is a transitive element of residual type for $CT(X)$.
Shimomura in \cite{S1}, \cite{S2}, \cite{S3} has extended these results, as we will describe below.
If $f$ is a transitive element for $S$, it means that every element of $S$ can be uniformly approximated by a map conjugate to
$f$, i.e. it is $f$ up to a change of variables.

Notice that when we say that $f \in S \subset C_s(X)$ is a transitive element for $S$ we are referring to the adjoint action of
$H(X)$ on $S$. Meanwhile, $f$ itself defines a dynamical system on $X$ which may or may not be topologically transitive.

The analysis proceeds by using finite approximations for $f \in C_s(X)$ as in \cite{AHK}.

In general, if $\A$ is an open cover of a compact metric space $X$, we can regard the elements of $\A$ as providing a finite
approximation to $X$. Think of them as pixels covering $X$. We can represent $f \in C(X)$ by using $f^{\A}$, the set of
pairs $(U_1,U_2) \in \A \times \A $ such that $f(U_1)$ meets
$U_2$. Here the relationship is usually not that of a function, $f(U_1)$ usually meets several members of $\A$.
When $X$ is a Cantor space, we restrict attention to the case when $\A$ is a decomposition
of $X$ by clopen subsets.

There are different
ways to think about  this finite setup. Most authors, including Bernardes and Darji \cite{BD} and Shimomura in his papers,
 regard the elements of $\A$ as vertices of a directed graph with
the oriented edges the pairs in $f^{\A}$. The original function $f$ is conjugate to an inverse limit constructed from such graphs.
The projective Fra\"{i}ss\'{e} constructions of Kwiatkowska in \cite{K} lead to similar inverse limit constructions.

 I prefer the equivalent approach of regarding the set $f^{\A}$  as a relation on $\A$, see, e.g.
\cite{A1} Chapter 5. A relation on a set $A$ is just a subset $R \subset A \times A$.
It is called a \emph{surjective relation}\index{surjective relation} when each of the two coordinate projections maps it  onto $A$.
A function $f$ on $A$ is a special case
of a relation  such that $f(x) = \{ y \in A : (x,y) \in f \}$ is a singleton for every $x$ and, abusively, we write $f(x)$
for both the singleton set and its unique element. We can iterate relations, and various dynamics concepts extend to relations.
For us, a \emph{system} will be a pair $(A,R)$ with $A$ a compact metric space and $R$ a closed, surjective relation on $A$.
In Section \ref{relations} we review these relation concepts from \cite{A1}
and recall various dynamic constructions like inverse limits and
the sample path system of a relation. For example, given a relation on a finite set, the sample path system
is the subshift of finite type associated with the relation. For a continuous surjective map on a compact metric space, the
sample path system is the natural homeomorphism lift. After the introductory general remarks of this section,
all our spaces are assumed to be
either Cantor spaces or finite.

In Section \ref{representations} we describe the elementary properties of the representation procedure.
Instead of  finite decompositions $\A$ on a Cantor space $X$, our tool for building representations will be an
\emph{indexed partition}\index{indexed partition}\index{partition!indexed} $\al : X \tto I$, a continuous surjective map to a nonempty,
finite discrete space. Since such a function is
locally constant,  a partition is just a decomposition $\A^{\al}$
whose members have been indexed by the set $I$.

We pick out a collection
of finite sets to serve as the index sets.  Let $\I$ be the countable set of all nonempty, finite subsets of finite products of
of $\N$, the set of positive integers.  This ensures that if $\phi$ is a nonempty relation on $I \in \I$ then $\phi$ itself is
a member of $\I$.

If $f \in C_s(X)$ and $\al : X \tto I$ is a partition then
$$f^{\al} \quad =_{def} \quad (\al \times \al)(f) \hspace{2cm}$$
\index{$f^{\al}$} is a surjective relation on $I$. If $\phi = f^{\al}$, then we will say that \emph{$\phi$ represents $f$ via $\al$}.

If $\CC(X,\I)$\index{$\CC(X,\I)$} is the countable set of indexed partitions, equipped
with the discrete topology, then
then  $\Gamma_0 : C_s(X) \times \CC(X,\I) \to \I$\index{$\Gamma_0$} by $(f,\al) \mapsto f^{\al}$ is a locally constant map. Projecting away
from the second coordinate we obtain the \emph{representation relation}
$$\Gamma \quad  = \quad \{ \ (f , \phi) \in C_s(X) \times \I \  : \ \phi = f^{\al} \ \mbox{ for some } \ \al \in \CC(X,\I) \ \}.$$
\index{$\Gamma$}That is, $\Gamma(f)$ is the set of surjective relations which represent
$f$.

If $f, g \in C_s(X)$ and $g = h^{-1} \circ f \circ h$ with $h \in H(X)$ then as subsets of $X \times X$,
$f = (h \times h)(g)$. Thus, if $\phi$ represents $f$ via $\al$,
then $\phi$ represents $g $ via
$\al h =_{def} \al \circ h$.  It follows that $\Gamma^{-1}(\phi) \subset C_s(X)$ is open and conjugacy invariant.

For $f, g \in C_s(X)$ we write $f  \sim_{\al}  g$ if $\al f = \al g$, or, equivalently, if $f(x)$ and $g(x)$ lie in the
same element of $\A^{\al}$ for all $x$. If $f \sim_{\al} g$ then $f^{\al} = g^{\al}$.  On the other hand, if
$f^{\al} = g^{\al}$ then there exists $h \in H(X)$ such that $h \sim_{\al} 1_X$ and $h \circ f \circ h^{-1} \sim_{\al} g$.

 In Section \ref{characterizations} we obtain the results about representations.  For example, Theorem \ref{theo3.02}:
\vspace{.25cm}

{\bfseries Theorem:} \emph{If  $\bt : X \tto J$  is a partition,  $\phi$ is a surjective relation on $I$,
and $f \in C_s(X)$
then there exist  $g \in H(X)$ and $\al : X \tto I$ such that}
$$g \ \sim_{\bt} \ f \qquad \mbox{ and } \qquad g^{\al} \ = \ \phi.$$
\vspace{.25cm}

This implies that $\Gamma^{-1}(\phi)$ is dense in $C_s(X)$ for every surjective relation $\phi$ on a set in $\I$.
It follows that the functions $f \in C_s(X)$ which can be represented by every surjective relation form a
conjugacy invariant, dense, $G_{\dl}$ subset. This subset is exactly the set of transitive elements for $C_s(X)$.

In general,  a closed, conjugacy invariant subset $K$ of $C_s(X)$
is characterized by the set $\Gamma(K)$ of relations in $\I$ which represent its members. Furthermore, the closure of the conjugacy
class of $f \in C_s(X)$ consists of exactly those $g \in C_s(X)$ such that $\Gamma(g) \subset \Gamma(f)$.

We call a closed, conjugacy invariant set $K$ a \emph{conjugacy transitive set} when it is the closure of a
single conjugacy class.  That is, the restriction of the action of the group $H(X)$ to $K$ is topologically transitive.
In that case, $Trans(K) \ = \ \{ f : \overline{H(X)\cdot f} = K \}$, the set of 
\emph{conjugacy transitive elements}\index{conjugacy transitive element} of $K$,
form a $G_{\dl}$ subset of $K$ which is, of course, dense.

We conclude the section by describing a lifting property introduced in \cite{AGW} and extended in \cite{S3}.
This property and an equivalent factoring property provide a sufficient condition that the conjugacy class of
an element $f \in C_s(X)$ be a $G_{\dl}$ set and show, in particular, that $f$ is then a homeomorphism. We say that
such a homeomorphism is of \emph{residual type}. If $K$ is a closed, conjugacy transitive subset, then $Trans(K)$
contains at most one conjugacy class whose elements are of residual type.

Finally, in Section \ref{examples} we collect the results built using the tools of the preceding sections. We show that
$C_s(X), C_s(X;1), CR(X),$ $ CR(X;1), CT(X)$ are conjugacy transitive subsets each of which admits a conjugacy transitive
element of residual type. On the other hand, $CM(X)$ and $CM(X;1)$ are conjugacy transitive subsets for which I can
construct no conjugacy transitive element of residual type and I conjecture that they do not exist.

The set $CM(X;1)$ is
of special interest. It is, in fact, conjugacy minimal. This follows from a theorem of Shimomura that if $f \in C_s(X)$ is
not a periodic function then the closure of the conjugacy class of $f$ contains $CM(X;1)$. This also shows that the only
other conjugacy minimal subset of $C_s(X)$ is the singleton $\{ 1_X \}$. There is a simple inverse limit construction which
yields elements $CM(X;1)$ in a $G_{\dl}$ subset which we label $H(X;1!)$. $f \in H(X;1!)$ if it is a homeomorphism with a
unique fixed point and for every $x \in X$, other than the fixed point,  the $\pm$ orbit $\{ f^k(x) : k \in \Z \}$ is dense in $X$.
The construction is associated with a little semigroup, the analysis of which shows that some members of $H(X;1!)$ are
topologically mixing while others are not even weak mixing. The weak mixing elements of $H(X;1!)$ form a dense $G_{\dl}$ subset
of $CM(X;1)$.
\vspace{1cm}

\section{Relations and Maps}\label{relations}
\vspace{.5cm}

Our spaces are metric spaces with metrics (all labeled $d$) bounded by $1$.
On a finite product $A_1 \times \dots \times A_n$ we use the
metric $d = max_{i=1}^n \ d \circ (\pi_i \times \pi_i)$. On a countably infinite product $\Pi_{i= 1}^{\infty} \ A_i$
we use $d = max_{i=1}^{\infty} \ \frac{1}{2^{i-1}}d \circ (\pi_i \times \pi_i)$.
Here $\pi_i$ is the $i^{th}$ coordinate projection. On a finite set or a discrete space like $\Z$, the set of
integers, or like $\N$, the set of positive integers, we use the zero-one metric.

We will use the relation notation following \cite{A1} and so we briefly review it.

For sets $A, B$ a  relation $R : A \to B$ is a  subset of
 $A \times B$. $R$ is a \emph{relation on $A$} when $B = A$. A map is a
 relation such that $R(x) = \{ y : (x,y) \in R \}$ is a singleton set for every
 $x \in A$.  For $A_0 \subset A$, the image $R(A_0) = \bigcup_{x \in A_0} \ R(x)$.
 $R(A_0)$ is the projection to $B$ of $R \cap (A_0 \times B) \subset A \times B$.
 $R^{-1} : B \to A$ is defined to be $\{ (y,x) : (x,y) \in R \}$. For $B_0 \subset B$
 we let $R^*(B_0) = \{ x \in A : R(x) \subset B_0 \} = A \setminus R^{-1}(B \setminus B_0)$.
 So $R^*(B_0) \subset R^{-1}(B_0) \cup R^*(\emptyset)$.
 If $R$ is a map, then $R^*(B_0) = R^{-1}(B_0)$.  The relation is called \emph{surjective}\index{surjective relation}\index{relation!surjective} when
 $R(A) = B$ and $R^{-1}(B) = A$, or, equivalently, $R(x) \not= \emptyset$ and $R^{-1}(y) \not= \emptyset$
 for all $x \in A, \ y \in B$.

  If $R : A \to B$ and $S : B \to C$ then the composition $S \circ R : A \to C$ is the image under the
 projection to $A \times C$ of the set $(R \times C) \cap (A \times S) \subset A \times B \times C$.
 Composition is associative.

 For a relation $R$ on $A$ we let $R^{n+1} = R^n \circ R$ for $n = 1,2,...$
 and let $R^0 = 1_A$ and we define $|R| = \{ x : (x,x) \in R \}$.
 We let $Per(R) = \{ n \in \N : |R^n| \not=\emptyset \}$\index{$Per(R)$}. If $A_0 \subset A$ then
 $R \cap (A_0 \times A_0)$ is the \emph{restriction} of $R$ to $A_0$.

A relation $F$ on $A$ is \emph{reflexive} when $1_X \subset F$, where $1_X$ is the identity map on $X$.
The relation is \emph{symmetric} when $F = F^{-1}$.  We will say that
$F$\emph{satisfies transitivity} when $F \circ F \subset F$.  Then $F \cap F^{-1}$ satisfies symmetry and
 transitivity and restricts
 to an equivalence relation on $|F|$. We call the equivalence classes in $|F|$ the \emph{basic sets} of $F$. We don't call
 $F$ a transitive relation because we give the latter term a different, dynamic, meaning, see below.

 {\bf N. B.} From now on we will assume that our spaces $A$ are  metric spaces which are either compact or discrete.

A closed relation $R$  is a closed subset of $A \times B$. From compactness various properties of closed relations follow.
 A map is continuous iff it is a closed relation.  Furthermore, the composition of closed relations is
 closed and the image of a closed set by a
 closed relation is closed. So if $B_0$ is open in $B$ then $R^*(B_0)$ is open in $A$.
If $R$ is a closed relation on
 $A$ then $|R|$ is a closed subset of $A$. If $F$ is a closed relation which satisfies transitivity,
 then the basic sets $\{ F(x) \cap F^{-1}(x) : x \in |F| \}$ are closed sets.

 For a relation $R$ on $A$, we define the following relations associated with $R$:
 \begin{itemize}
 \item  The \emph{orbit relation} is $\OO R = \bigcup_{n = 1}^{\infty} \ R^n$.\index{orbit relation}\index{$\OO R$}

 \item The \emph{wandering  relation} is $\NN R = \ol{\OO R}$\index{wandering relation}\index{$\NN R$}

 \item The  \emph{chain relation}  $ \CC R = \bigcap_{\ep > 0} \ \OO (V_{\ep} \circ R \circ V_{\ep})$ where \\
 $V_{\ep} = \{ (x,y) : d(x,y) < \ep \}$.\index{chain relation}\index{$\CC R$}
 \end{itemize}
 The relations $\NN R$ and $\CC R$ are closed and the relations $\OO R$ and $\CC R$ satisfy transitivity.
 Clearly,  $R \subset \OO R \subset \NN R \subset \CC R$ and if $A$ is discrete,
 then  $\OO R = \NN R = \CC R$.

 For a closed relation $R$ on $A$, we will say that $R$ is \emph{transitive}\index{transitive} when $\OO R = A \times A$,
 \emph{topologically transitive}\index{topologically transitive}\index{transitive!topologically}  when $\NN R = A \times A$, and
 \emph{chain transitive} when \index{chain transitive}\index{transitive!chain}
 $\CC R = A \times A$. We call $R$ \emph{periodic} when $|R^n| = A$ for some $n \in \N$.
 We call $R$ \emph{recurrent}\index{recurrent} when $|\OO R| = A$, \emph{topologically recurrent}\index{topologically recurrent}\index{recurrent!topologcially}
 when $|\NN R| = A$ and \emph{chain recurrent}\index{chain recurrent}\index{recurrent!chain} when $|\CC R| = A$. We will refer to these as the three
 \emph{transitivity properties} and the three \emph{recurrence properties}.

 For subsets $U,V \subset A$ the \emph{hitting time set}index{hitting time set} is
 $$N_R(U,V) = \{ n \in \N : R^n(U) \cap V \not= \emptyset \} = \{ n \in \N : R^n \cap (U \times V) \not= \emptyset \}.$$
 \index{$N_R(U,V)$}
 A closed relation $R$ is topologically transitive iff for all nonempty open $U,V \subset A$, $N_R(U,V) \not= \emptyset$.
 It is topologically recurrent iff for all nonempty open $U \subset A$, $N_R(U,U) \not= \emptyset$.

 If $A$ is discrete, then for a surjective relation $R$ the three recurrence properties coincide and the
 three transitivity properties coincide.

 A  relation $R$ on $A$ is called \emph{mixing}\index{mixing} if there exists a positive integer
 $N$ such that $R^N = A \times A$. It then follows, by induction, that $R^n = A \times A$ for every
 $n \geq N$.  To see this, observe that if $R$ is mixing, then it is surjective and so,
 given $x_1, x_2 \in A$, there exists $x_3 \in A$ such that
 $x_3 \in R(x_1)$. By inductive hypothesis $x_2 \in R^n(x_3)$ and so $x_2 \in R^{n+1}(x_1)$.
$R$ is called \emph{chain mixing}\index{chain mixing}\index{mixing!chain} if for every $\ep > 0$  the relation
 $V_{\ep} \circ R \circ V_{\ep}$ is mixing. $R$ is called \emph{topologically mixing} if for every $\ep > 0$ there
 exists a positive integer $N$ such that $n \geq N$ implies $V_{\ep} \circ R^n \circ V_{\ep} = A \times A$. Of course,
 topological mixing implies chain mixing. Again, if $A$ is discrete, these three concepts coincide. We refer to these
 as the three \emph{mixing properties}. By the nine \emph{dynamics properties} we mean the three recurrence, the
 three transitive and the three mixing properties.

If $R$ is a closed, surjective relation on $A$ and $S$ is a closed relation on $A$ with $R \subset S$ then $S$ is surjective
 and it satisfies any of the nine dynamics properties when  $R$ satisfies the corresponding property.
 Notice that if $S$ is a map then $R \subset S$ implies
 $R = S$ since $R$ is assumed surjective.

  If $R$ is a closed, surjective relation on $A$ with $A$ compact, we let and $A_R \subset A^{\Z}$ be the \emph{sample path space}\index{sample path space}
with shift homeomorphism $S_R$ on $A_R$ and the projection map $p_0 : A_R \to A$.
\begin{equation}\label{1.01aa}
\begin{split}
A_R \ = \ \{ z \in A^{\Z} : (z_n,z_{n+1}) \in R \ \mbox{ for all} \ n \in \Z \},\\
S_R(z)_n \ = \ z_{n+1}, \qquad p_0(z) \ = \ z_0. \hspace{1cm}
\end{split}
\end{equation}
The projection $p_0$ is a continuous surjection with
$(p_0 \times p_0)(S_R) = R$. Also, $Per(S_R) = Per(R)$.

\begin{prop}\label{prop1.01} Let $R$ be a closed surjective relation on a finite set $A$.
$$ R \quad \mbox{recurrent} \ \  \Longrightarrow \qquad (A_R,S_R) \qquad \mbox{has dense periodic points}.$$
$$ R \quad \mbox{transitive} \ \ \ \Longrightarrow \qquad (A_R,S_R) \qquad \mbox{is topologically transitive}.$$
$$ R \quad \mbox{mixing} \qquad \ \Longrightarrow \qquad \ (A_R,S_R) \qquad \mbox{is topologically mixing}. \hspace{.4cm}$$
\end{prop}

\begin{proof} If $R$ is recurrent then $A$ is a union of basic sets for $R$ and if $z \in A_R$ then
all $z_i \in A$ lie in the same basic set.  For any positive integer $N$ there is a sequence
$z_N = u_0,....,u_{k+1} = z_{-N}$ with $(u_i,u_{i+1}) \in R$ for $0 \leq i \leq k$.
The word $z_{-N}....z_{N}u_1...u_k$ on $[-N,N + k]$ extends to a periodic point in $A_R$ which agrees with
$z$ on $[-N,N] $.

If $R$ is transitive then $A$ consists of a single basic set and so $A_R$ contains a point $z$ in which every
finite word of $A_R$ occurs, i.e. $z$ is a transitive point.

If $R^M = A \times A$ and $z, w \in A_R$ then for any positive integer $N$ there is a sequence
$z_N = u_0,...,u_{M} = w_{-N}$. So there is an element $q$ of $A_R$ which agrees with $z$ on $(-\infty,N]$
and $S^{M + N}(q)$ agrees with $w$ on $[-N, \infty)$. Thus, $(A_R, S_R)$ is topologically mixing.
\end{proof} \vspace{.5cm}

We will write $(A,R)$ for a pair consisting of a compact metric space $A$ and closed relation $R$ on $A$. We will
call the pair a \emph{dynamical system}\index{dynamical system} or just a system, when the relation $R$ is surjective.

We will say that $p$ is a map of systems, and write $p : (A_1,R_1) \to (A_2,R_2)$ when $p : A_1 \to A_2$ is
a continuous map with $(p \times p)(R_1) \subset R_2$, or, equivalently $R_1 \circ p \subset p \circ R_2$. Since $R_1$ is
surjective,
the latter inclusion is an equality  if  $R_2$ is a map. If $R_1$ and $R_2$ are maps, then $p : (A_1,R_1) \to (A_2,R_2)$
 iff the following diagram commutes:
 \[ \begin{CD}
A_1 \ @> R_1 >> \ A_1\\
 @V{p}VV              @VV{p}V\\
A_2 \ @> R_2 >> \ A_2.
  \end{CD}  \]

  In general, $p : (A_1,R_1) \to (A_2,R_2)$ implies
 $p : (A_1,R^n_1) \to (A_2,R^n_2)$
 for $n \in \Z, \ p :(A_1,\A R_1) \to (A_2,\A R_2)$ for $\A = \OO, \NN, \CC$ and $p(|R_1|) \subset |R_2|$.
It follows that $Per(R_1) \subset Per(R_2)$.

 We will write $p : (A_1,R_1) \tto (A_2,R_2)$ when $p : A_1 \tto A_2$ is
a continuous surjection with $(p \times p)(R_1) = R_2$. We will then say that $p$ is surjective or that
$p$ maps $R_1$ onto $R_2$ or that $(A_2,R_2)$ is a factor of $(A_1,R_1)$. In general,
\begin{equation}\label{1.01cc}
\begin{split}
 (p \times p)(R_1) \ =  \ p \circ R_1 \circ p^{-1},\qquad \mbox{ and so} \hspace{1cm} \\
(p \times p)(R_1) \ = \ R_2 \qquad \Longleftrightarrow \qquad  p \circ R_1 \circ p^{-1} \ = \ R_2.
\end{split}
\end{equation}

For example, for every $i \in \Z$, the projection $p_i: (A_R, S_R) \tto (A,R)$ maps the
sample space homeomorphism $S_R$ back
onto the relation $R$ itself.

In the special case when $R = f $ is a continuous map, we label the sample path space pair $(X_f,S_f)$
as $(\hat X, \hat f)$ and call it the \emph{natural lift} of $f$ to a homeomorphism. If $z \in \hat {X}$ then
$z_{i+n} = f^n(z_i)$ for all $i \in \Z$ and $n \in \N$.  That is, the $i^{th}$ coordinate determines all of the
later coordinates.

\begin{lem}\label{lem1.01a} (a) Let $(A_i,R_i)$ be systems for $i = 1,2,3$. If $p : (A_1, R_1) \tto (A_2, R_2)$ and
 $q : A_2  \to A_3$ then
$q \circ p : (A_1, R_1) \tto (A_3,R_3)$ iff $q : (A_2,R_2) \tto (A_3,R_3)$.

(b)Let $p : (A_1,R_1) \to (A_2,R_2)$ with $p(A_1) = A_2$, i.e. $p$ is surjective between the
underlying spaces. If for each $(a,b) \in R_2$ either $R_2(a) = \{ b \}$ or $R_2^{-1}(b) = \{ a \}$ then
$(p \times p)(R_1) = R_2$, i.e. $p : (A_1,R_1) \tto (A_2,R_2)$. \end{lem}

\begin{proof} (a): $qp (A_1) = q[p(A_1)] = q(A_2)$. Hence, $qp(A_1) = A_3$ iff $q(A_2) = A_3$.
Similarly, $ (qp \times qp)(R_1) =
(q \times q)[(p \times p)(R_1)] = (q \times q)(R_2)$. Hence, $ (qp \times qp)(R_1) = R_3$ iff
$ (q \times q)(R_2) = R_3$.

(b): Assume $(a,b) \in R_2$ and $R(a) = \{ b \}$. There exists $a_1 \in A_1$ such that
$p(a_1) = a$ and there exists $b_1 $ such that $(a_1,b_1) \in R_1$ because $R_1$ is surjective. Since
$p \times p$ maps $R_1$ into $R_2$, $(a,p(b_1)) \in R_2$. Since $R(a)$ is a singleton, $p(b_1) = b$ and so
$(a,b) = (p \times p)(a_1,b_1)$.  If, instead $R^{-1}(b) = \{ a \}$ then first lift $b$ and proceed as before.
\end{proof}\vspace{.5cm}

If $C $ is a closed subset of $A$ and $R$ is a closed relation on $A$ then the \emph{restriction} $R|C$ of $R$ to
$C$ is given by $R\cap (C \times C)$. When $(A,R)$ is a system, i. e. $R$ is surjective, then $(C,R|C)$ is a
\emph{subsystem} when $R|C$ is a surjective relation on $C$, i.e. when $C \subset R^{-1}(C) \cap R(C)$. For a system $(A,R)$
the closed set $C \subset A$
is called an transitive, topologically transitive, or chain transitive subset
when the restriction $(C,R|C)$ is a subsystem which satisfies the
corresponding property. Similarly, for the three recurrence properties and the three mixing properties.

If $p : (A,R) \to (B,S)$ and $C \subset A$ and $(C,R|C)$ is a subsystem of $(A,R)$ then
$(p(C),(p \times p)(R|C))$ is a system which satisfies any of the dynamics properties that $(C,R|C)$ does.
Since $(p \times p)(R|C) \subset S|p(C)$ it follows that $(p(C),S|p(C))$ is a subsystem satisfying the same dynamics properties.

 For compact spaces, an \emph{inverse sequence of spaces}\index{inverse sequence of spaces} is a
 sequence of continuous surjections $\{ p_{n+1,n} : A_{n+1} \tto A_n : n \in \N \}$.
 For $m > n$ we define $p_{m,n} : A_m \tto A_n$ to be the composition  $p_{n+1,n} \circ \cdots \circ p_{m,m-1}$.
 The \emph{inverse limit}\index{inverse limit} $A_{\infty}$ is defined by
\begin{equation}\label{1.01bb}
A_{\infty} \  = \  \{ \ z \in \Pi_{n \in \N} \ A_n \ : \ z_n = p_{n+1,n}(z_{n+1}) \quad \mbox{for all} \ n \in \N \ \}.
\end{equation}
 The surjection $p_n : A_{\infty} \tto A_n$ is the restriction of the projection to the $n^{th}$ coordinate.

 If $f$ is a surjective map on $X$ then we can let $A_n = X, p_{n+1,n} = f$ for all $n$. The resulting inverse limit space
 is just a relabeling of $\hat X$, the natural lift of $f$. To be precise, if $w \in X_{\infty}$ we define $z \in \hat X$ by
 $z_{-n} = w_n$ for all $n \in \N$ and $z_k = f^{k+1}(w_1)$ for $k = 0,1,\dots$.

We will say that the sequence \emph{bifurcates}\index{sequence bifurcates} when for every $n \in \N$ and $x \in A_n$ there exist $m > n$ and
$y_1 \not= y_2 \in A_m$ such that $p_{m,n}(y_1) = x = p_{m,n}(y_2) $.  It is clear that if the sequence bifurcates then
the limit space $A_{\infty}$ is perfect, i.e. it has no isolated points.  Conversely, if all the $A_n$'s
are finite and $A_{\infty}$ is
perfect (and so is a Cantor set), then the sequence bifurcates.

If all of the spaces $A_n$ are perfect then $A_{\infty}$ is perfect  whether the sequence bifurcates or not. In particular,
if $f $ is a surjective continuous map on a Cantor set $X$, then the natural homeomorphic lift $\hat X$ is
perfect and so is a Cantor set as well.

A map from $\{ p_{n+1,n} : A_{n+1} \tto A_n \}$ to $\{ q_{n+1,n} : B_{n+1} \tto B_n \}$ is a sequence $\{ h_n : A_n \to B_n \}$
of continuous maps such that $$q_{n+1,n} \circ h_{n+1} = h_n \circ p_{n+1,n}$$ for all $n \in \N$. By restricting
$\Pi_n \ h_n : \Pi_n \ A_n \to \Pi_n B_n $ we obtain
a continuous map $h_{\infty} : A_{\infty} \to B_{\infty}$ with
$q_n \circ h_{\infty} = h_n \circ p_n$ for all $n \in \N$. If each $h_n$ is surjective then for
any $z \in B_{\infty}$, $ \{ (q_n \circ h_{\infty})^{-1}(z) \}$ is a decreasing sequence of nonempty compact sets
and the nonempty intersection is $ (h_{\infty})^{-1}(z)$.  Hence, $h_{\infty}$ is surjective when the $h_n$'s are.

 By identifying $A_1 \times A_1 \times A_2 \times A_2$ with $A_1 \times A_2 \times A_1 \times A_2$
 we define  $(A_1 \times A_2, R_1 \times R_2)$. Then $Per(R_1 \times R_2) = Per(R_1) \cap Per(R_2)$.  The similar
 identification of $(A_1 \times A_2)^{\Z} $ with $A_1^{\Z} \times A_2^{\Z}$ identifies $(A_{R_1 \times R_2}, S_{R_1 \times R_2})$
 with $(A_{R_1} \times A_{R_2},S_{R_1} \times S_{R_2})$.

 Similarly, for inverse sequences $\{ p_{n+1,n} : A_{n+1} \tto A_n \}$ and $\{ q_{n+1,n} : B_{n+1} \tto B_n \}$
 we can naturally identify $(A \times B)_{\infty}$, the inverse limit of
 $\{ p_{n+1,n} \times q_{n+1,n}  : A_{n+1} \times B_{n+1} \tto A_n \times B_n \} $, with $A_{\infty} \times B_{\infty}$.

 An \emph{inverse sequence of systems}\index{inverse sequence of systems} is a sequence of continuous surjections
 $\{ p_{n+1,n} : (A_{n+1},R_{n+1}) \tto (A_n,R_n) : n \in \N \}$. With the above identifications, $R_{\infty}$, the
 inverse limit of $\{ p_{n+1,n} \times p_{n+1,n} : R_{n+1} \tto R_n \}$, is a closed surjective relation on $A_{\infty}$.
 So we say that the system $(A_{\infty},R_{\infty})$ is the limit of the inverse sequence of systems.
 The \emph{reverse system} $\{ p_{n+1,n} : (A_{n+1},R^{-1}_{n+1}) \tto (A_n,R^{-1}_n) : n \in \N \}$ has limit
$(A_{\infty},R^{-1}_{\infty})$.

\begin{df}\label{df1.02} An inverse sequence  of systems $\{ p_{n+1,n} : (A_{n+1},R_{n+1}) \tto (A_n,R_n) : n \in \N \}$ satisfies
the \emph{Shimomura Condition}\index{Shimomura Condition}  if for every $n \in \N$ there exists $m > n$ so that
$p_{m,n} \circ R_m : A_m \tto A_n$ is a mapping. \end{df}
\vspace{.5cm}

\begin{prop}\label{prop1.03} Let $\{ p_{n+1,n} : (A_{n+1},R_{n+1}) \tto (A_n,R_n) : n \in \N \}$ be an inverse sequence of systems.
\begin{itemize}
\item[(a)] If the sequence satisfies the Shimomura Condition then $R_{\infty}$ is a surjective continuous map on $A_{\infty}$.

\item[(b)] If the sequence and its reverse both satisfy the Shimomura Condition then $R_{\infty}$ is a homeomorphism on $A_{\infty}$.

\item[(c)] If for every $n \in \N$, $(a,b), (a,c) \in R_{n+1}$ implies $p_{n+1,n}(b) = p_{n+1,n}(c)$ then the sequence
satisfies the Shimomura Condition.

\item[(d)] If each $A_n$ is finite and $R_{\infty}$ is a mapping, then the sequence satisfies the Shimomura Condition.
\end{itemize}
\end{prop}

\begin{proof}  (a): Assume $(x,u),(x,v) \in R_{\infty}$.  To show that $u = v$ it suffices to show that $u_n = v_n$ for
every $n \in \N$. Let $m > n$ be such that $p_{m,n} \circ R_m $ is a map.  Since $(x_m,u_m), (x_m,v_m) \in R_m$,
it follows that $u_n = p_{m,n}(u_m) = p_{m,n}(v_m) = v_n$. Hence, the surjective closed relation $R_{\infty}$ is a
surjective continuous map.

(b): If the the sequence and its reverse satisfy the condition, then $R_{\infty}$ and $R_{\infty}^{-1}$ are mappings. This implies
that $R_{\infty}$ is a homeomorphism with inverse $R_{\infty}^{-1}$.

(c):  The assumption says that $p_{n+1,n} \circ R_{n+1}$ is a mapping.

(d): Call $(a,b),(a,c) \in R_n$ with $b \not= c$ a \emph{V} in $R_n$. If the Shimomura Condition fails then there exists $n \in \N$
so that for every $m > n$ there exists a V in $R_m$ which projects to a V in $R_n$.
If $A_n$ is finite, there are only
finitely many V's in $R_n$. Hence, there exists a V in $R_n$ which is an image of V's in $R_m$ for $m$ arbitrarily large.
By compactness there is a V $(z,u),(z,v) \in R_{\infty}$ which projects to a V in $R_n$. Hence, $R_{\infty}$ is not a mapping.
\end{proof}\vspace{.5cm}

If $f$ is a surjective map on $X$ then with $A_n = X, p_{n+1,n} = f$, we obtain the inverse sequence of systems
with $(A_n,R_n) = (X,f)$ for all $n$ whose inverse limit is the natural lift $(\hat X,\hat f)$, as described above.
Notice that $f \circ f$ and
$f \circ f^{-1} = 1_X$ are maps and so the sequence and its reverse satisfy the Shimomura Condition.

If $(A,R)$ is a system and $N$ is a positive integer we define the \emph{$N$-fold discrete suspension}\index{discrete suspension}
$(A_N,R_N)$ with $A_N = A \times \{ 1, \dots, N \}$  and
\begin{equation}\label{1.02}
R_N = \{ \ ((a,i),(a,i+1)) \ : \ a \in A, i < N \ \} \cup \{ \ ((a,N),(b,0)) \ : \ (a,b) \in R \ \}.
\end{equation}
Clearly, $R_N$ is a map, or a homeomorphism, when $R$ is. Notice that
\begin{equation}\label{1.03}
((a,i),(b,j)) \in (R_N)^N \qquad \Longleftrightarrow \qquad (a,b) \in R \quad \mbox{and} \quad j = i.
\end{equation}

We can identify the successive suspensions $((X_M)_N,(R_M)_N)$ with $(X_{MN},R_{MN})$ by
\begin{equation}\label{1.04}
((x,i),j) \ \mapsto \ (x, j + N(i-1)). \hspace{3cm}
\end{equation}

The construction is functorial. If $p : (A,R) \to (B,S)$ then $p_N : (A_N,R_N) \to (B_N,S_N)$ where
$p_N = p \times 1_{\{1, \dots, N\}}$.  If $p$ is surjective then $p_N$ is.
Hence, if $\{ p_{n+1,n} : (A_{n+1},R_{n+1}) \tto (A_n,R_n)  \}$
is an inverse system then so is $\{ (p_{n+1,n})_N : ((A_{n+1})_N,(R_{n+1})_N) \tto ((A_n)_N,(R_n)_N) \}$. We leave to the reader
the easy proof of the following.

\begin{lem}\label{lem1.04} With the obvious identifications, the inverse limit of
$\{ (p_{n+1,n})_N : ((A_{n+1})_N,(R_{n+1})_N) \tto ((A_n)_N,(R_n)_N) \}$ becomes
$((A_{\infty})_N,(R_{\infty})_N)$. If the original sequence bifurcates or satisfies the
Shimomura Condition, then so does the $N$-fold suspension sequence. \end{lem}

$\Box$ \vspace{.5cm}

For a continuous map $f$ on $A$ with $A$ compact, there are different definitions for topological transitivity, see \cite{AC}.
As defined above, we say a map $f$ is
 topologically transitive when $\NN f = A \times A$, or, equivalently, the \emph{hitting time set}
 $N_f(U,V) \not= \emptyset $ for all open,nonempty $U, V \subset A$. It is equivalent to $Trans_f \not= \emptyset$
 when $x \in Trans_f$ if $\{ f^i(x) : i \in \N \}$ is dense in $X$. A topologically transitive map is surjective.

The map is \emph{minimal}\index{minimal map}\index{map!minimal} when $Trans_f = A$. Equivalently, for every nonempty open $U \subset A$ the sequence
of open sets $\{ f^{-n}(U) : n \in \N \}$ covers $A$ (and so has a finite subcover).

We write $C_s(A,B)$ for the space of continuous surjections from $A$ to $B$,
  $H(A)$ for the homeomorphism group and $C_s(A)$ for the space of continuous surjections on $A$.  All are
 equipped with the sup metrics. We let $ CR(A), CT(A), CM(A)$\index{$CR(A)$}\index{$CT(A)$}\index{$CM(A)$} denote 
 the subsets of chain recurrent, chain transitive
  and chain mixing surjective continuous
 mappings on $A$. We let $TT(A)$\index{$TT(A)$} and $MM(A)$\index{$MM(A)$} denote the subsets of topologically transitive and minimal maps on $A$.
 A map $f \in C_s(A)$ is \emph{weak mixing}\index{mixing!weak} when $f \times f \in TT(A \times A)$. We let $WM(A)$\index{$WM(A)$} denote the weak mixing
 maps on $A$.

 \begin{prop}\label{prop1.05} Let $A$ be a compact metric space with $C(A)$  the space of continuous maps on $A$.

 (a) The sets $C_s(A), CR(A), CM(A)$ and $C(A;Per \supset Q) = \{ f : |f^n| \not= \emptyset$ for all $n \in Q \}$\index{$C(A;Per \supset Q)$}
 are closed subsets of $C(A)$ for $Q$ any subset of $\N$.

 (b) The sets $H(A), TT(A), MM(A), WM(A)$ and $\{ f \in H(X) : f(x) = x $ or $ \overline{\{ f^i(x) : i \in \Z \}} = X$ for all
 $x \in A \}$ are $G_{\dl}$ subsets of $C(A)$.  The set of $f \in C(A)$ which admit exactly one fixed point is a $G_{\dl}$
 subset of $C(A)$.
 \end{prop}

 \begin{proof} (a)  If $U$ is a proper subset of $A$ then the condition $f(A) \subset U$ is an open condition by
 compactness. Hence, $C_s(A)$ is closed.
 Similarly, if $U$ is an open subset of $A \times A$ then $\CC f \subset U$ implies that
 $\overline{\OO (V_{\ep} \circ f \circ V_{\ep})} \subset U$ for
 some $\ep > 0$. It thus follows that $\CC f \subset U$ is an open condition. Hence, $CR(A) =
 \{ f : 1_A \subset \CC f \}$ and $CT(A) = \{ f : A \times A \subset \CC f \} $ are closed in $C_s(A)$, see also \cite{A1} Chapter 7.
 By \cite{A1} Exercise 8.22 it follows that a chain transitive $f$ is not chain mixing iff it factors over a nontrivial periodic
 orbit, see also \cite{RW}.  This last is an open condition and so $CM(A)$ is closed as well. The condition $|f^n| = \emptyset$
 is equivalent to
 $inf_{x \in A} \ d(f^n(x),x) \ > \ 0$ and this is an open condition. It follows that $C(A;Per \supset Q)$ is closed.

 (b) For every $\ep > 0$ the condition $(f \times f)[(X \times X) \setminus V_{\ep} \subset (X \times X) \setminus 1_X$ is an
 open condition. Intersecting over rational $\ep$ we see that the condition that $f$ be injective is $G_{\dl}$.

  For fixed open sets $U, V$ the condition
 $N_f(U,V) \not= \emptyset$ is an open condition. Intersecting over $U, V$ in a countable basis we see that
 $TT(A)$ is $G_{\dl}$.

 The map $q : C_s(A) \to C_s(A \times A)$ given by $f \mapsto f \times f$ is continuous. Hence, $WM(A) = q^{-1}(TT(A \times A))$
 is $G_{\dl}$.

 A map $f$ is minimal iff for every open $U \subset $ nonempty, there exists $L \in \N$ such that
  $X \subset \bigcup_{k = 0}^L \ f^{-k}(U)$.
 For each open set $U$, this is an open condition because it is
 equivalent to finding a closed finite cover $\{ A_0, \dots,A_N \}$ of $X$ such that
 for $i = 0,\dots, N, \ f^{k_i}(A_i) \subset U$, for some $0 \leq k_i \leq L$. Intersecting over a countable basis we see
 that $f$ minimal is a $G_{\dl}$ condition.

 Given $\ep > 0$, let $ G(f,\ep) = \{ x : d(f(x),x)) < \ep \}$. It is easy to check that for a homeomorphism $f$,
 the condition that every point $x$ is
 either fixed by $f$ or has a dense $\pm$ orbit is equivalent to the condition that for every $\ep > 0$ and every
 nonempty open $U \subset A$ there exists $L \in \N$ such that $X \subset G(f,\ep) \ \cup \  \bigcup_{k = -L}^L \ f^{-k}(U)$.
 For each open $U$ and $\ep > 0$ this is an open condition as above. Notice that if $A$ is a closed set then
 $A \subset G(f,\ep)$ iff $sup_{x \in A} \ d(f(x),x) \ < \ \ep$.  Intersecting over rational $\ep$ and $U$ in a countable basis
 we obtain a $G_{\dl}$ condition.

 Given $\ep \geq 0$, let $ K(f,\ep) = \{ x : d(f(x),x)) \leq \ep \}$ so that $K(f,0) = |f|$. Observe that if $U$ is any open set
 which contains $K(f,\ep)$ for some $\ep \geq 0$ then there exists $\ep_1 > \ep$ such that $K(f,\ep_1) \subset U$. A mapping
 $f$ has at most one fixed point iff for every $\ep > 0$ there exists $\dl > 0$ so that $x_1, x_2 \in K(f,\dl)$ implies
 $d(x_1,x_2) < \ep$. That is, if $K(f,\dl) \times K(f,\dl) \subset V_{\ep}$. For each $\ep > 0$ and $\dl > 0$ fixed this is
 an open condition on $f$. Taking the union over $\dl > 0$ and then the intersection over rational $\ep > 0$ we obtain
 the $G_{\dl}$ condition. Intersect with the closed set of $f$ such that $|f| \not= \emptyset$ and we obtain the $G_{\dl}$
 set of maps which admit a unique fixed point.
 \end{proof}\vspace{.5cm}

 {\bfseries Remark:} There are different definitions for topological transitivity, see \cite{AC}. We call a map $f$
 topologically transitive when $\NN f = A \times A$ which is equivalent to nonempty hitting times sets $N(U,V)$ as described
 above. It is equivalent to $Trans_f \not= \emptyset$ when $x \in Trans_f$ if $\{ f^i(x) : i \in \N \}$ is dense in $X$.
 On the other hand, if $(A,f)$ is obtained by compactifying $\Z$ so that the translation map $n \mapsto n + 1$ extends to
 the homeomorphism $f$ on $A$ which is not topologically transitive in this sense although it does have a dense $\pm$ orbit.
 Happily, when the space is perfect then all definitions agree. For example, if $A$ is perfect and the $\pm$ orbit of $x$ is
 dense then either $x \in Trans_f$ or $x \in Trans_{f^{-1}}$. To see this observe that $x$ is either in the closure
 of $\{ f^i(x) : i \in \N \}$ or $\{ f^{-i}(x): i \in \N \}$.  Suppose the first, then there is a sequence $i_n \to \infty$
 such that $f^{i_n}(x) \to x$.  So for every $k \in \N$, $f^{i_n - k}(x) \to f^{-k}(x)$.  Thus, the entire $\pm$ orbit is in the
 closure of the forward orbit and so $x \in Trans_f$.  Otherwise, $x \in Trans_{f^{-1}}$ which implies that $f^{-1}$ is topologically
 transitive. Since the inverse of a topologically transitive homeomorphism is topologically transitive (e.g.
 $\NN (f^{-1}) = (\NN f)^{-1}$), it follows from either case that $f$ is topologically transitive.
 \vspace{.5cm}

 If $1 \in Per(f)$ then $Per(f) = \N$. We denote will write  $C_s(X; 1)$ for $C_s(X; Per \supset \{ 1 \})$, the set of surjective
 maps which admit a fixed point. Similarly, we will write $H(X;1)$ for $ H(X) \cap C_s(X;1)$ and
 $ CM(X; 1)$ for $CM(X) \cap C_s(X;1)$ etc. We will write $H(X; 1!)$ for the set of homeomorphisms $f$ such that
 $f$ has a unique fixed point and if $x \in X$ is not the fixed point then the $\pm$ orbit $\{ f^i(x) : i \in \Z \}$ is dense in $X$.
\vspace{.5cm}

%
%
%

 We will apply all this to two special cases: $A = X$ a \emph{Cantor space}, a Cantor set equipped with an ultrametric $d$, i.e.
 $$d(x,y) \quad \leq \quad max(d(x,z),d(z,y)) \hspace{2cm}$$ or to $A$ a finite set with the zero-one metric, also an ultrametric.
 Notice that if $A$ is a finite set
 then the metrics $d = max_{i = 1}^{N} \ d \circ (\pi_i \times \pi_i)$ on  $A^{N}$ and
 $d = max_{i \in \Z} \ \frac{1}{2^{|i|}}d  \circ (\pi_i \times \pi_i)$ on $A^{\Z}$ are ultrametrics.
%
  With $d$ an ultrametric $V_{\ep} = \{ (x,y) : d(x,y) < \ep \}$ is a
 clopen equivalence relation.

 \vspace{1cm}

 \section{Representations of Mappings via Indexed Partitions}\label{representations}
\vspace{.5cm}

 Our spaces $X,X_1, X_2,$ etc are all Cantor spaces, i.e. nonempty, zero-dimensional, perfect, compact metric spaces
 equipped with ultra-metrics.
 The maps are assumed to be continuous. We repeatedly use
 \emph{The Uniqueness of Cantor}\index{Uniqueness of Cantor}, the observation that all Cantor spaces are homeomorphic and, in particular,
 as nonempty clopen subsets of a Cantor space are Cantor spaces they
  are all homeomorphic to one another. A \emph{decomposition}\index{decomposition} $\A$ of $X$ is a finite, pairwise disjoint cover of $X$ by
 nonempty clopen subsets. Since the metric $d$ on $X$ is an ultrametric, the set of balls $\{ V_{\ep}(x) : x \in X \}$ is a
 decomposition of $X$ for any $\ep > 0$.

 Let $\I$ denote the countable set of all nonempty, finite subsets of $\N^n$ for $n = 1,2,\dots$. Notice that if $\phi : I_1 \to I_2$
 is a nonempty relation between  elements of $\I$ then $\phi \in \I$. We regard $\I$ as a discrete set and the elements of $\I$ as
 finite discrete spaces. Recall that a discrete space uses the zero-one metric $d$.

 An \emph{indexed partition} (hereafter, just
 a \emph{partition})\index{partition}\index{partition!indexed}
 is a continuous surjection $\al : X \tto I$ with $I \in \I$. We define
 $\A^{\al} = \{ A^i = \al^{-1}(i) : i \in I \}$ to be the \emph{associated decomposition}\index{decomposition!associated}
 of $X$.  If $\A_1, \A_2$ are decompositions then $\A_1$ \emph{refines}\index{decomposition!refinement} $\A_2$ iff for every
 $A_1 \in \A_1$ there exists a -necessarily unique- $A_2 \in \A_2$ such that $A_1 \subset A_2$.
 We will say of partitions $\al_1 : X \tto I_1, \al_2 : X \tto I_2$ that $\al_1$ refines $\al_2$ when
 $\A^{\al_1}$ refines $\A^{\al_2}$.

If $\A$ is any decomposition of $X$ and $I \in \I$ has the same cardinality as $\A$ then
 there is a partition $\al : X \tto I$ such that $\A^{\al} = \A$. Since there are decompositions of any positive finite
 cardinality, it follows that for any $I \in \I$  there exist  partitions $\al : X \tto  I$.

 Let $\CC(X,\I)$ denote the set of all indexed  partitions. Each clopen subset of $X$ is a finite union of basic sets and
 so there are only countably many clopen sets. Hence,for each $I \in \I$ there are only countably many partitions
 $\al : X \tto I$.  Since $\I$ is countable, it follows that $\CC(X,\I)$ is countable.  We give it the discrete topology.

 If $\al_1 : X_1 \tto I_1$ and $\al_2 : X_2 \tto I_2$ are  partitions
 then $\al_1 \times \al_2 : X_1 \times X_2 \tto I_1 \times I_2$ is a  partition. If
 $X_1 = X_2 = X$ then we write $\al_1 \otimes \al_2 : X \to I_1 \times I_2$ by $\al_1 \otimes \al_2 (x) = (\al_1(x),\al_2(x))$.
 This is usually not surjective and so is only an indexed partition
 when we restrict the range to $(\al_1 \otimes \al_2)(X) \subset I_1 \times I_2$.

The \emph{mesh}\index{mesh} of a finite collection $\A$ of sets is  $max \ diam \{ A : A \in \A \}$. For a partition $\al : X \to I$
the mesh of $\A^{\al}$, also called the \emph{mesh} of $\al$, is
\begin{equation}\label{2.01}
max \{ d(x,y) : \al(x) = \al(y)  \} \quad = \quad inf \{ \ep : (\al \times \al)^{-1}(1_I) \subset V_{\ep} \}.
\end{equation}
 The \emph{thickness}\index{thickness} of $\al$ is the minimum of the diameters, i. e.
 $$ min \ \{ \ max \{ d(x,y) : \al(x) = \al(y) = i \} \ : \ i \in I \ \}.$$
 The \emph{Lebesgue number}\index{Lebesgue number} of $\al$ is
 \begin{equation}\label{2.02}
 min \{ d(x,y) : \al(x) \not= \al(y) \} \quad = \quad max \ \{ \ep : V_{\ep} \subset (\al \times \al)^{-1}(1_I) \}.
 \end{equation}
 Note that $ (\al \times \al)^{-1}(1_I) = \bigcup_{i \in I} \ A^i \times A^i $ is a clopen neighborhood of the diagonal.
 If a set has diameter
 less than the Lebesgue number of $\al$ then $\al$ is constant on it and so it is contained in an element of $\A^{\al}$.

 \begin{prop}\label{prop2.01} Let $\al : X \tto I$ and $\bt : X \tto J$ be partitions.
 \begin{enumerate}
 \item[(a)] The following are equivalent.
 \begin{itemize}
 \item[(i)] $\al$ refines $\bt $.
 \item[(ii)] There exists  $\pi : I \tto J$ such that $\pi \circ \al = \bt $.
 \item[(iii)] The relation $\bt \circ \al^{-1} : I \to J$ is a map.
 \item[(iv)] $\bt $ is constant on every element of $\A^{\al}$.
 \end{itemize}
When these conditions hold, $\pi = \bt \circ \al^{-1}$ is the unique map such that $\pi \circ \al = \bt $.

 \item[(b)] If the mesh of
 $\al$ is less than Lebesgue number of $\bt $, then $\al$ refines $\bt $. If, in addition,
 $mesh \ \al \ < \ thickness \ \bt $ then each element of $\A^{\bt}$ contains at least two elements of $\A^{\al}$, or, equivalently,
 $\al$ is not constant on any element of $A^{\bt }$.
 \end{enumerate}
 \end{prop}

  \begin{proof} (a) (i) $\Rightarrow$ (ii): Define $\pi(i) = j$ if $\al^{-1}(i) \subset \bt^{-1}(j)$.

  (ii) $\Rightarrow$ (i): If $\pi(i) = j$ and  $\pi \circ \al = \bt $ then
  $\al^{-1}(i) \subset  (\pi \circ \al)^{-1}(j) = \bt^{-1}(j)$.

  (ii) $\Rightarrow$ (iii): Because $\al$ is a surjective map, $\al \circ \al^{-1} = 1_I$ and so
 $\pi \circ \al = \bt $ implies $\pi =  \pi \circ \al \circ \al^{-1} = \bt \circ \al^{-1}$.

 (iii) $\Rightarrow$ (ii): Because $\al$ is a map $1_X \subset  \al^{-1} \circ \al$. Hence,
 $\bt \subset \bt \circ \al^{-1} \circ \al$. As the composition of maps, the latter is a map
 and inclusion between maps implies equality. Hence, $\pi \circ \al = \bt $ with $\pi = \bt \circ \al^{-1}$.

 (iii) $\Leftrightarrow$ (iv): Obvious.

 (b)  The first part is clear from the definition of the Lebesgue number.

 If $B \in \A^{\bt }$ contains a unique $A \in \A^{\al}$ then since $\A^{\al}$ refines $\A^{\bt }$, $A = B$. Then
 $mesh \ \al \ \geq \ diam \ A \ = \ diam \ B \ \geq \ thickness \ \bt $.
 \end{proof}\vspace{.5cm}

 \begin{prop}\label{prop2.02}  Let $I, J \in \I$.

 (a) If $\al : X \tto I$ is a partition and $\pi : J \tto I$, then there exists $\bt : X \tto J$ such that $\pi \circ \bt = \al$.

 (b)  If $\al_1 : X_1 \tto I$ and $\al_2 : X_2 \tto I$ are partitions
  then there exists
  a homeomorphism $h :X_1 \to X_2$ such that $\al_1 = \al_2 \circ h$.
 \end{prop}

 \begin{proof} (a): For each $i \in I$, $\al^{-1}(i)$ is a nonempty clopen subset of $X$ and so is a Cantor space.
 Since $\pi^{-1}(i)$ is a nonempty subset of $J$, it is an element of $\I$. We can
 choose $\bt : \al^{-1}(i) \tto \pi^{-1}(i)$ for all $i \in I$ and concatenate.

 (b): Choose a homeomorphism between the Cantor spaces $h : \al_1^{-1}(i) \to \al_2^{-1}(i) $ for all $i \in I$ and concatenate.
\end{proof} \vspace{.5cm}

  If $\al: X_2 \tto I$ is a partition and $f : X_1 \to X_2$ is a continuous map
 then we write $\al f : X_1 \to I$ for $\al \circ f$.  If $f$ is surjective, i.e. $f \in C_s(X)$, then $\al f$ is
 a partition.

 If $\al : X_2 \tto I$ is a partition then for continuous maps $f, g : X_1  \to X_2$
 we write $f \sim_{\al} g$ when $\al  f = \al  g$, or, equivalently,
 if $f(x) $ and $g(x)$ lie in the same member of the decomposition $\A^{\al}$ of $X_2$ for all $x \in X_1$.
  This defines a clopen equivalence relation
 on $C_s(X_1,X_2)$. Clearly, $f \sim_{\al} g$ implies $d(f,g) \ \leq \ mesh \ \al$ and so we can use these equivalence
 relations to measure closeness of approximations between maps in $C(X)$.

 \begin{cor}\label{cor2.03} If  $\al : X_2 \tto
 I$ is a partition and $p : X_1 \tto X_2$, then there
 exists a homeomorphism $h : X_1 \to X_2$ such that $p \sim_{\al} h$. \end{cor}

 \begin{proof} Apply Proposition \ref{prop2.02}(b) with $\al_2 = \al$ and $\al_1 = \al \circ p$.
\end{proof} \vspace{.5cm}

 {\bfseries Remark:}  With $X_1 = X_2 = X$ it follows that $H(X)$ is dense in $C_s(X)$. 
\vspace{.5cm}

 For $f$ a closed relation on $X$ and a partition $\al : X \tto  I$ , we let $f^{\al} = (\al \times \al)(f) \subset I \times I$.
 So $(i,j) \in f^{\al}$ iff there exists $(x,y) \in f$ such that $\al(x) = i$ and $\al(y) = j$ and so
 iff $A^i \cap f^{-1}(A^j) \not= \emptyset$.
  Clearly, if the relation $f$ is a surjective relation on $X$ then
 $f^{\al}$ is a surjective relation on $I$ and  $\al : (X,f) \tto (I,f^{\al})$ is a surjective system map.

 If $f$ is a surjective map, i.e. $f \in C_s(X)$, then
 \begin{equation}\label{2.02a}
 (\al \otimes \al f)(X) \ = \ f^{\al}, \hspace{3cm}
 \end{equation}
 and so $\al \otimes \al f : X \tto f^{\al}$ is a partition.

 \begin{prop}\label{prop2.04} If $\al : X \tto I$ is a partition and $\ep > 0$ is less than the Lebesgue number of $\al$,
then  $f^{\al} = (V_{\ep} \circ f \circ V_{\ep})^{\al}$ for any closed surjective relation $f$ on $X$.
 \end{prop}

 \begin{proof} If $\ep$ is less than the Lebesgue number then the clopen equivalence relation $V_{\ep}$
 is contained in $(\al \times \al)^{-1}(1_I)$.

 If $(x,y) \in f$ and $(x,x_1), (y,y_1) \in V_{\ep}$ then $\al(x) = \al(x_1)$ and $\al(y) = \al(y_1)$.  Hence,
 $(\al \times \al)(f) = (\al \times \al)(V_{\ep} \circ f \circ V_{\ep})$.
 \end{proof}\vspace{.5cm}

For any closed surjective relation $f$ on $X$, we have the surjective map of systems $\al : (X,f) \tto (I,f^{\al})$ and so
$\al : (X,f^n) \tto (I,(f^{\al})^n)$ for $n \in \Z, \
 \al : (X,\CC f) \tto (I,\OO f^{\al})$, and $Per(f) \subset Per(f^{\al})$. From this and Proposition \ref{prop2.04} we obtain

 \begin{prop}\label{prop2.05} If $f$ is a surjective closed relation on $X$ and $\al : X \tto I$ is a
 partition then
 \begin{equation}\label{2.03}
 \begin{split}
 f \quad \mbox{chain recurrent} \qquad \Longrightarrow \qquad f^{\al}  \quad \mbox{is recurrent}.\\
 f \quad \mbox{chain transitive} \qquad \Longrightarrow \qquad f^{\al}  \quad \mbox{is transitive}.\\
  f \quad \mbox{chain mixing} \qquad \Longrightarrow \qquad f^{\al}  \quad \mbox{is mixing}. \hspace{.5cm}
  \end{split}
  \end{equation}
  \end{prop}

 $\Box$\vspace{.5cm}
%
%

 For partitions $\al_1 : X_1 \tto I_1$ and $\al_2 : X_2 \tto I_2$, we write $(p,\pi) : \al_1 \tto \al_2$
 when $p : X_1 \tto X_2$ and  $\pi : I_1 \tto I_2$ such that the following diagram commutes:
 \[ \begin{CD}
 X_1 \ @> \al_1 >> \ I_1\\
 @V{p}VV              @VV{\pi}V\\
  X_2 \ @> \al_2 >> \ I_2
  \end{CD}  \]
  i.e. $ \pi \circ \al_1 \ = \al_2 \circ p $.

 \begin{prop}\label{prop2.07}  With $f_1 \in C_s(X_1)$ and $ f_2 \in C_s(X_2)$ assume that $p : (X_1,f_1) \tto (X_2,f_2)$
 is a surjective system map.
 If $\al_1 : X_1 \tto I_1$ and $\al_2 : X_2 \tto I_2$
 are partitions and $(p,\pi) : \al_1 \tto \al_2$.
 then
 \begin{equation}\label{2.06}
 \begin{split}
 (\pi \times \pi)\circ (\al_1 \otimes \al_1 f_1)  =
 (\pi \circ \al_1 \otimes \pi \circ \al_1 f_1)  =  (\al_2 \otimes \al_2 f_2) \circ p, \\
  \mbox{and so} \qquad (\pi \times \pi)(f_1^{\al_1}) \quad = \quad f_1^{\pi \circ \al_1} \quad = \quad f_2^{\al_2}. \hspace{2cm}
\end{split}
  \end{equation}
Furthermore, the following diagram of systems commutes:
  \[ \begin{CD}
(X_1,f_1) \ @> \al_1 >> \ (I_1, f_1^{\al_1})\\
 @V{p}VV              @VV{\pi}V\\
(X_2,f_2) \ @> \al_2 >> \ (I_2, f_1^{\al_2}).
  \end{CD}  \]
  \end{prop}

  \begin{proof} Because $(p,\pi)$ maps $\al_1$ to $\al_2$ we have $ \pi \circ \al_1 \ = \al_2 \circ p $
  and
 $ \pi \circ \al_1 f_1  = \al_2 \circ p \circ f_1  = \al_2  f_2 \circ p. $ This and (\ref{2.02a}) imply (\ref{2.06}).
 The commutative diagram of surjective system maps is then clear.
  \end{proof}\vspace{.5cm}

   \begin{prop}\label{prop2.07a} (a) Assume $f_1 \in C_s(X_1)$ and  $\al_1 : X_1 \tto I$ is a partition.
Let $h : X_1 \to X_2$  be a homeomorphism.
If $\ f_2 = h \circ f_1 \circ h^{-1} \in C_s(X_2)$
and $\al_2 = \al_1 \circ h^{-1} : X_2 \tto I$, then
\begin{equation}\label{2.04}
\begin{split}
 \al_2 \otimes \al_2  f_2 \ = \ (\al_1 \otimes \al_1  f_1) \circ h^{-1}. \hspace{2cm}\\
\mbox{and so} \qquad f_2^{\al_2} \ = \ f_1^{\al_1}. \hspace{3cm}
\end{split}
\end{equation}

   (b) Assume  $\al_1 : X_1 \tto I$ and $\al_2 : X_2 \tto I$ are partitions and that $f_1 \in C_s(X_1),
 f_2 \in C_s(X_2)$. If $f_1^{\al_1} = f_2^{\al_2}$ then there exists a homeomorphism $h : X_1 \to X_2$
 such that $\al_2 \circ h = \al_1$ and $h^{-1} f_2 h \sim_{\al_1} f_1$.
 \end{prop}

 \begin{proof}  (a) We have $h : (X_1,f_1) \tto (X_2,f_2)$ and $(h, 1_I) : \al_1 \tto \al_2$. So (\ref{2.04})
 follows from (\ref{2.06}).

 (b) Let $\phi = f_1^{\al_1} = f_2^{\al_2} \in \I$. $\al_1 \otimes \al_1 f_1 : X_1 \tto \phi$ and
 $\al_2 \otimes \al_2 f_2 : X_2 \tto \phi$ are partitions and so
 by Proposition \ref{prop2.02} (b) there is a homeomorphism $h : X_1 \to X_2$ such that
 $(\al_1 \otimes \al_1 f_1) = (\al_2\otimes \al_2 f_2) \circ h$. That is,
 $\al_1 = \al_2 \circ h$  and $(\al_2  f_2 )\circ h = \al_1 f_1$.
 Hence, $\al_1 \circ (h^{-1}f_2 h) = (\al_2 h)\circ (h^{-1}f_2h) = (\al_2  f_2 )\circ h = \al_1 f_1$ and so
 $h^{-1} f_2 h \sim_{\al_1} f_1$.
  \end{proof}\vspace{.5cm}

The following estimate was proved by Bermudez \& Darji \cite{BD} and  by Shimomura \cite{S3}.

 \begin{prop}\label{prop2.06} Let $f, g \in C_s(X)$. If $g^{\al} \subset f^{\al}$ then
 $d(f,g) \leq mesh \ \A^{\al} + mesh \ f \A^{\al}$ with $f \A^{\al} = \{ f(A) : A \in \A^{\al} \}$.
 \end{prop}

 \begin{proof} For all $x \in X, \ (\al(x),\al(g(x))) \in g^{\al} \subset f^{\al}$ and so there exists $y \in X$
 such that  $(\al(x),\al(g(x))) = (\al(y),\al(f(y))$.
 Since, $d(f(x),f(y)) \leq mesh \ f A^{\al}$ and $d(f(y),g(x))$ $ \leq mesh \ \A^{\al}$ the result follows from the
 triangle inequality.
 \end{proof} \vspace{.5cm}

 If $f \sim_{\al} g$ for  $f,g \in C_s(X)$ then $\al \otimes \al f = \al \otimes \al g$ and so
 $f^{\al} = g^{\al}$ by (\ref{2.02a}). Hence, with the discrete topology on $\CC(X,\I)$ the map
 $\Gamma_0 : C_s(X) \times \CC(X,\I) \to \I$ by $(f,\al) \mapsto f^{\al}$ is  locally constant. We can regard $\Gamma_0$ as the
 set of triples $\{ (f,\al,\phi) : f^{\al} = \phi \}$ and project away from the second coordinate to define
 \begin{equation}\label{2.05}
 \Gamma \  =_{def} \  \{ (f,\phi) : f^{\al} = \phi \ \ \mbox{for some} \ \al \in \CC(X,\I) \} \ \subset \ C_s(X) \times  \I.
 \end{equation}\index{$\Gamma$}

  We will say that  $\phi$, a surjective relation on $I \in \I$,
  \emph{represents}\index{representing $f$} $f \in C_s(X)$ if there exists a partition $\al$ such that $f^{\al} = \phi$, i.e.
  if $\phi \in \Gamma(f)$. Equivalently, $\phi$ represents $f$ if there exists a system surjection
  $\al : (X,f) \tto (I,\phi)$.


From Proposition \ref{prop2.07} we obtain the following.

\begin{cor}\label{cor2.08} Let $f \in C_s(X)$.
(a) If $\pi : (I,\phi) \tto (I_1,\phi_1)$   and $\phi \in \Gamma(f)$, then $\phi_1 \in \Gamma(f)$.

(b) If $p : (X,f) \tto (X_1,f_1)$ then $\Gamma(f_1) \subset \Gamma(f)$.

(c) If $(\hat X, \hat f)$ is the natural lift of $f$ to a homeomorphism then $\Gamma(f) \ = \ \Gamma(\hat f)$. \end{cor}

\begin{proof} (a): If $\al : X \tto I$ is a partition with $f^{\al} = \phi$ and $\al_1 = \pi \circ \al$ then
$(1_X,\pi) : \al \tto \al_1$ and so
Proposition \ref{prop2.07}  implies $f^{\al_1} = \phi_1$.

(b): If $\al_1 : X_1 \tto I$ is a partition with $f_1^{\al_1} = \phi$ and $\al = \al_1 \circ p$, then
$(p,1_I) : \al \tto \al_1$, and so Proposition \ref{prop2.07}  implies $f^{\al} = \phi$.

(c): Since $p_0$ maps $(\hat X, \hat f)$ onto $(X,f)$, it follows from (b) that $\Gamma(f)  \subset  \Gamma(\hat f)$.

Now let $\al : \hat X \tto I$ be a partition and $\phi = \hat f^{\al}$. Because every clopen set is a finite union of basic sets
it follows that there is a finite list of coordinates so that $\al(z)$ depends only on the value of $z$ at each of these coordinates.
Furthermore, if $i$ is the smallest index in this set of coordinates then  $z_n = f^{n-i}(z_i)$ if  $n \geq i$.
This means that $\al(z)$ is a function of $z_i = p_i(z)$. That is, there exists a partition $\bt : X \tto I$ such that
$\al = \bt \circ p_i$.  Because $p_i : (\hat X, \hat f) \tto (X,f)$ and $(p_i,1_I): \al \tto \bt$,
Proposition \ref{prop2.07} implies that
$\phi = \hat f^{\al} = f^{\bt}$.  Hence, $\phi \in \Gamma(f)$.
\end{proof}\vspace{.5cm}

Because it is of interest to know whether a system $(X,f)$ factors over a nontrivial periodic orbit, we observe the following.

\begin{prop}\label{prop2.09} For $f \in C_s(X)$ and a permutation $\gamma$ on $J \in \I$ the following are equivalent.
\begin{itemize}
\item[(i)] $f$ factors over the permutation $\gamma$, i.e. there is a system surjection $\bt : (X,f) \tto (J, \gamma)$.
\item[(ii)] $\gamma \in \Gamma(f)$.
\item[(iii)] There exists $\phi \in \Gamma(f)$ which factors over $\gamma$.
\end{itemize}
\end{prop}

\begin{proof} (i) $\Leftrightarrow$ (ii) and (ii) $\Rightarrow$ (iii) are obvious.

If $\al : X \tto I$ with $f^{\al} = \phi$ and $\pi : (I,\phi) \tto (J, \gamma)$ then
$\al : (X,f) \tto (I, \phi)$ and $\pi \circ \al : (X,f) \tto (J, \gamma)$ are surjective system maps.
\end{proof} \vspace{1cm}

\section{Representation Characterizations}\label{characterizations}

In this section we justify our emphasis on the use of surjective relations on finite sets to study $C_s(X)$ and $H(X)$
for $X$ a Cantor space. We will show that a closed, conjugacy invariant subset of $C_s(X)$ is characterized by the
set of relations which represent it, i.e. by $\Gamma(K)$. We  develop the inverse limit construction which will be used
in studying the examples in the next section. Finally, we describe a useful sufficient when the limit of such a
construction has a $G_{\dl}$ conjugacy class and so is a residual element of the closure of its conjugacy class.

We will use homeomorphisms to identify the various Cantor spaces which turn up in our constructions
so that we regard all our partitions as lying in $\CC(X, \I)$.  If $f_1 \in C_s(X_1), f_2 \in C_s(X_2)$ we will say
they are \emph{conjugate} if there exists a homeomorphism $h :X_1 \to X_2$ such that $h\circ f_1 \circ h^{-1}$ is in
the orbit of $f_2$ with respect to the $H(X_2)$ adjoint action on $C_s(X_2)$. This definition is, of course, independent
of the choice of homeomorphism $h$ which is used to identify the two spaces. Equivalently, $f_1$ and $f_2$ are conjugate if there
exists a homeomorphism $h$ such that  $h\circ f_1 \circ h^{-1} \ = \ f_2$.

We first use the sample path space construction to show that every surjective relation $\phi$ on an element $I$ of $\I$ can be
represented by an element of $H(X)$. Usually, the sample path system of the finite system
$(I,\phi)$ will do, but a bit of extra work is needed to assure that the associated space is Cantor.

 \begin{theo}\label{theo3.01} If $\al : X \tto I$ is a partition and $\phi$ is a surjective relation on $I$ then
 there exists  $f \in H(X)$ such that $f^{\al} = \phi$ and $Per(f) = Per(\phi)$.

 If  $\phi$ is recurrent then $f$ can be chosen with dense periodic points.

 If $\phi$ is transitive then
 $f$ can be chosen topologically transitive with dense periodic points.

 If $\phi$ is mixing then
 $f$ can be chosen topologically mixing with dense periodic points.
\end{theo}

 \begin{proof} On $[2] = \{ 0,1 \}$ let $U = [2] \times [2]$ so that $([2]_U,S_U)$ is the full shift with $[2]_U$  the
 Cantor space $[2]^{\Z}$.   Let $\phi_1 = \phi \times U$ on $I \times [2] = I_1$.
 Let $(X_1, f_1) = ((I_1)_{\phi_1}, S_{\phi_1})$, the sample path system for $\phi_1$. Observe that
  $X_1 = I_{\phi} \times [2]_U$ is a Cantor space and that $(f_1)^{p_0} = \phi_1$. Furthermore,
  $Per(f_1) = Per(\phi) \cap Per(U) = Per(\phi)$. See the description
 associated with Proposition \ref{prop1.01}.

  With $\pi : I_1 \tto I$ the first
  coordinate  projection, we define $\al_1 : \pi \circ p_0 : X_1 \tto I$. Since $(1_{X_1},\pi) : p_0 \tto \al_1$,
 (\ref{2.06}) implies $ f_1^{\al_1} =  (\pi \times\pi) ( (f_1)^{p_0} ) \ =  \ (\pi \times\pi) (\phi_1) = \phi$.

  By Proposition \ref{prop2.02} (b) there exists a homeomorphism
  $h : X_1  \to X$ so that $\al \circ h = \al_1$. Let $f = h f_1 h^{-1}$.
  Since $h$ maps $f_1$ to $f$,  (\ref{2.04}) implies that $f^{\al} = \phi$. Clearly $Per(f) = Per(f_1)$.

  The transitivity and mixing results follow from Proposition \ref{prop1.01}.
  \end{proof}\vspace{.5cm}

Next, we show that any element of $ C_s(X)$ can be approximated by a homeomorphism which is represented by $\phi$.

\begin{theo}\label{theo3.02} If  $\bt : X \tto J$  is a partition,  $\phi$ is a surjective relation on $I$,
and $f \in C_s(X)$
then there exists  $g \in H(X)$ such that
\begin{itemize}
\item  $g \ \sim_{\bt } \ f$.

\item There exists a partition $\al : X \tto I$ such that $g^{\al} = \phi$.

\item $Per(f) \cap Per(\phi) \subset Per(g)$.

\item If $f$ is chain recurrent and $\phi$ is recurrent then $g$ can be chosen with dense periodic points.


\item If $f$ is chain mixing and $\phi$ is mixing then $g$ can be chosen topologically mixing and with dense periodic points.
\end{itemize}
\end{theo}

\begin{proof} Apply Proposition \ref{prop2.02}(a) to get a partition $\al_1 : X \tto I \times J$ such that
$\pi_2 \circ \al_1 = \bt $. Define $\al_2 = \pi_1 \circ \al_1 : X \tto I$ and $\phi_1 = \phi \times f^{\bt }$.
By Theorem \ref{theo3.01} there
exists a homeomorphism $g_1$ on $X$ such that $g_1^{\al_1} = \phi_1$. By (\ref{2.06})
\begin{equation}\label{3.01a}
\begin{split}
g_1^{\bt} \ = \ (\pi_2 \times \pi_2)(\phi_1) \ = \ f^{\bt }, \hspace{1cm}\\
g_1^{\al_2} \ = \ (\pi_1 \times \pi_1)(\phi_1) \ = \ \phi, \hspace{1cm}
\end{split}
\end{equation}
 By Proposition \ref{prop2.07a}(b) there exists $h \in H(X)$ such that
$h^{-1} \circ g_1 \circ h \sim_{\bt } f$.  Let $g = h^{-1} \circ g_1 \circ h $ and let  $\al = \al_2 \circ h$.
By (\ref{2.04}), $g^{\al} = g_1^{\al_2} = \phi$.

The homeomorphism $g_1$ can be chosen so that $Per(g_1) = Per(\phi_1) = Per(f^{\bt }) \cap Per(\phi) \supset Per(f) \cap Per(\phi)$.

If $f$ is chain recurrent then $f^{\bt }$ is recurrent.  If, in addition, $\phi$ is recurrent then the product $\phi_1$ is recurrent.

If $f$ is chain mixing then $f^{\bt }$ is mixing.  If, in addition, $\phi$ is mixing then the product $\phi_1$ is mixing.

Hence, by Theorem \ref{theo3.01} we can choose $g_1$ as required and so the conjugate $g$ satisfies the additional properties
as well.
\end{proof}\vspace{.5cm}

\begin{cor}\label{cor3.02a} If $I \in \I$ and $\phi$ a surjective relation on $I$, then $\Gamma^{-1}(\phi)$ is an
open, conjugacy invariant,  dense subset of $C_s(X)$. Furthermore,
\begin{itemize}
\item[(i)] $H(X) \cap \Gamma^{-1}(\phi)$ is dense in $H(X)$ and open relative to $H(X)$.

\item[(ii)] If $\phi$ is mixing then $\{ f \in H(X): f \in \Gamma^{-1}(\phi) \ $
and $f$ is topologically mixing with dense periodic points $ \} $ is dense
in $CM(X)$.

\item[(iii)] If $\phi$ is recurrent then $\{ f \in H(X) : f \in \Gamma^{-1}(\phi) \ $
and $f$ has dense periodic points $ \} $ is dense
in $CR(X)$.

\item[(iv)] $\{ f \in H(X) : f \in \Gamma^{-1}(\phi) \ $ and $Per(\phi) \subset Per(f) \ \}$ is
dense in $\{ f \in C_s(X) : Per(\phi) \subset Per(f) \ \}$.
\end{itemize}
\end{cor}

\begin{proof} $\Gamma^{-1}(\phi)$ is open because the map $\Gamma_0 : C_s(X) \times \CC(X,\I) \to \I$ is locally
constant.  It is conjugacy invariant by (\ref{2.04}).  The rest then follows from Theorem \ref{theo3.02}.
\end{proof}\vspace{.5cm}

We turn now the the inverse limit constructions.

\begin{df}\label{df3.03} We say that an inverse sequence of systems
$\{ \ p_{n+1,n} : (I_{n+1},\phi_{n+1}) \tto (I_n,\phi_n) \ : \ n \in \N \ \}$ with $I_n \in \I$ for all $n$, is
a \emph{Shimomura Sequence}\index{Shimomura Sequence} when it is bifurcating and satisfies the Shimomura Condition. It is an
\emph{invertible Shimomura Sequence}\index{Shimomura Sequence!invertible} when, 
in addition, $\{ \ p_{n+1,n} : (I_{n+1},\phi^{-1}_{n+1}) \tto (I_n,\phi^{-1}_n) \}$
satisfies the Shimomura Condition as well. \end{df}
\vspace{.5cm}

\begin{theo}\label{theo3.04} If $(X,f)$ is the inverse limit of a Shimomura Sequence
$\{ \ p_{n+1,n} : (I_{n+1},\phi_{n+1}) \tto (I_n,\phi_n) \ \}$, then $X$ is a Cantor space with
$f \in C_s(X)$, and if the sequence is an invertible Shimomura Sequence, then $f \in H(X)$.

$\Gamma(f)$ consists of the factors of $\{ \phi_n \}$.  That is, a surjective relation $\phi$ on $I \in \I$ lies
in $\Gamma(f)$ iff there exists $n \in \N$ and $\pi : I_n \tto I$ with $\phi = (\pi \times \pi)(\phi_n)$. \end{theo}

\begin{proof} $X$ is a Cantor space because the sequence bifurcates. The surjective relation $f$ on $X$ is
a map by Proposition \ref{prop1.03} which also says that $f$ is a homeomorphism if the sequence is invertible.

If $p_n : X \tto I_n$ is the projection from the limit then $f^{p_n} = $ \\ $(p_n \times p_n)(f) = \phi_n$.
Hence, $\phi_n \in \Gamma(f)$.

If $\bt : X \tto I$ is a partition, then
because $X$ is the inverse limit, and $I$ is finite, $\bt $ factors through $p_n : X \tto I_n$ for sufficiently large $n$.
That is, there exists $n \in \N$ and $\pi : I_n \tto I$ such that $\bt = \pi \circ p_n$. By ( \ref{2.06})
$f^{\bt } = (\pi \times \pi)(\phi_n)$ and so $f^{\bt}$ is a factor of $\phi_n$.

For the converse, Corollary \ref{cor2.08} (a) implies that any factor of an
element of $\Gamma(f)$ lies in $\Gamma(f)$.
\end{proof}\vspace{.5cm}

\begin{df}\label{df3.05} If $(I,\phi), (I_1,\phi_1)$ are surjective relations on elements of $\I$, we call
$\phi_1$ a \emph{$+$directional lift}\index{lift!$+$directional} of $\phi$ if there exists $\pi : I_1 \tto I$ such that
\begin{itemize}
\item[(i)] $(\pi \times \pi)(\phi_1) \ = \ \phi$.

\item[(ii)] $\pi \circ \phi_1 : I_1 \tto I$ is a map.

\item[(iii)] $\pi^{-1}(i)$ contains more than one element for every $i \in I$.
\end{itemize}
We then say that $\pi$ \emph{induces the lift}\index{lift!induced}  or the lift occurs \emph{via} $\pi$.

It is a \emph{$\pm$directional lift}\index{lift!$\pm$directional}  if $\pi$ satisfies, in addition,
\begin{itemize}
\item[(iv)]   $\pi \circ \phi^{-1}_1 : I_1 \tto I$ is a map,
\end{itemize}
or, equivalently, if, in addition, $\phi_1^{-1}$ a $+$directional lift of $\phi^{-1}$ via $\pi$.
\end{df}
\vspace{.5cm}

Notice that if $\pi: (I_1,\phi_1) \tto (I,\phi)$ induces a $+$ directional lift and $\pi_2 : (I_2,\phi_2) \tto (I_1,\phi_1)$ and
$\pi_3 : (I,\phi) \tto (I_3,\phi_3)$ are surjections with $I_2, I_3 \in \I$, then $\pi_3 \circ \pi \circ \pi_2$ induces
a $+$ directional lift $(I_2,\phi_2)$ of $(I_3,\phi_3)$.

Thus, an inverse sequence $\{ \ p_{n+1,n} : (I_{n+1},\phi_{n+1}) \tto (I_n,\phi_n) \} $
is a Shimomura sequence exactly when $I_n \in \I$ for all $n$ and for every $n$ there exists
$m > n$ such that $p_{m,n}$ induces a $+$directional lift, from $\phi_n$ to $\phi_{m}$.

%

\begin{lem}\label{lem3.06} Let $\al : X \tto I, \al_1 : X \tto I_1$ be partitions and $f \in C_s(X)$.
If $mesh \ \al_1$ is less than the thickness of $\al$ and $\al_1$ refines both $\al$ and $\al f$ then
$f^{\al_1}$ is a $+$ directional lift of $f^{\al}$ induced by the unique surjective map $\pi = \al \circ \al_1^{-1}$
such that $\al = \pi \circ \al_1$.

If $f \in H(X)$ and, in addition, $\al_1$ refines $\al (f^{-1})$ then
$f^{\al_1}$ is a $\pm$ directional lift of $f^{\al}$ induced by $\pi$.
\end{lem}

\begin{proof} If $\al_1$ refines $\al$, then Proposition \ref{prop2.01}(a) implies that
$\pi = \al \circ \al_1^{-1}$ is the unique surjective map such that
such that $\al = \pi \circ \al_1$ and (\ref{2.06}) implies  $(\pi \times \pi)(f^{\al_1}) = f^{\al}$.

If  $mesh \ \al_1$ is less than the thickness of $\al$, Proposition \ref{prop2.01}(b) implies that
each member of $\A^{\al}$ contains more than one element of $\A^{\al_1}$ and so no $\pi^{-1}(i)$ is a singleton.

$ f^{\al_1}  = (\al_1 \times \al_1)(f) = \al_1 \circ f \circ \al_1^{-1}$ (see (\ref{1.01cc})) and so
$\pi \circ f^{\al_1} \ = \ ( \al  f )\circ \al_1^{-1}$.
By Proposition \ref{prop2.01}(a) again this is a map if $\al_1$ refines $\al  f$.

Thus, from the hypotheses it follows that $\pi$ induces a $+$ directional lift from $f^{\al}$ to $f^{\al_1}$.

If $f \in H(X)$ we apply the result to $f^{-1}$ to obtain a $\pm$ directional lift when, in addition,
$\al_1$ refines $\al ( f^{-1})$.
 \end{proof}\vspace{.5cm}

 We call $\{ \ \al_n : X \tto I_n \ : \ n \in \N \ \}$ a \emph{basic sequence of partitions} when
 \begin{itemize}
 \item $\al_{n+1}$ refines $\al_n$ for all $n \in \N$.
\item $mesh \ \al_n \ \to \ 0$ as $n \ \to \ \infty$.
\end{itemize}
For a basic sequence, let $p_{n+1,n} : I_{n+1} \tto I_n$ be $\al_n \circ \al_{n+1}^{-1}$, the unique
surjection such that $\al_n = p_{n+1,n} \circ a_{n+1}$.

\begin{theo}\label{theo3.07} Let $\{ \ \al_n : X \tto I_n \ \}$ be a basic sequence of partitions. If
$f \in C_s(X)$ then $\{ \ p_{n+1,n} : (I_{n+1},f^{\al_{n+1}}) \tto (I_n,f^{\al_n}) \}$ is
a Shimomura sequence with inverse limit $(X,f)$. If $f \in H(X)$ then it is a invertible Shimomura Sequence.
\end{theo}

\begin{proof} For every $n$, if $m$ is sufficiently large, then
$mesh \ \al_m$ is less than the minimum of the Lebesgue number of
$\al_n$, the Lebesgue number of $\al_n f$ and the thickness of $\al_n$.  By Lemma \ref{lem3.06} $p_{m,n} : I_m \tto I_n$
induces a $+$directional lift from $f^{\al_n}$ to $f^{\al_m}$. It follows that the sequence is a Shimomura Sequence.

The surjections $\al_n : (X,f) \tto (I_n,f^{\al_n})$ induce a surjection from $(X,f)$ onto the inverse limit.
because the $mesh \ \al_n \to 0$, the map is injective and so is an isomorphism from $(X,f)$ to the inverse limit.

If $f \in H(X)$ we can choose $m$ large enough that $\al_m$ refines $\al_n (f^{-1})$ as well showing that the
sequence is a invertible Shimomura Sequence.
\end{proof}\vspace{.5cm}

\begin{cor}\label{cor3.08} If $\{ \ \al_n : X \tto I_n \ \}$ is a basic sequence of partitions, and
$f \in C_s(X)$, then $\Gamma(f)$ consists of the factors of $\{ f^{\al_n} \}$. \end{cor}

\begin{proof} This is immediate from Theorems \ref{theo3.07} and \ref{theo3.04}.
\end{proof}\vspace{.5cm}

If $\QQ$ is a collection of surjective relations on elements
of $\I$, we let
\begin{equation}\label{3.02a}
\Gamma^*(\QQ) \ = \ \{ f \in C_s(X) : \Gamma(f) \subset \QQ \}.
\end{equation}\index{$\Gamma^*$}
That is, $f \in \Gamma^*(\QQ)$ when all the relations which represent $f$ are  in $\QQ$.

\begin{df}\label{df3.09} Let $\QQ$ be a nonempty collection of surjective relations on elements
of $\I$.  We say that $\QQ$ \emph{satisfies Condition $\Gamma$}\index{Condition $\Gamma$}\index{$\Gamma$!Condition}  if
\begin{itemize}
\item $\QQ$ is closed under factors. That is, if $\psi$ is a factor of $\phi \in \QQ$ then $\psi \in \QQ$.
\item $\QQ$ admits $+$ directional lifts. That is, if $\phi \in \QQ$ then
there exists $\psi \in \QQ$ which is a $+$ directional lift of $\phi$.
\end{itemize}
\end{df}
\vspace{.5cm}

\begin{prop}\label{prop3.10} If $f \in C_s(X)$ then $\Gamma(f)$ is closed under factors and if
$\phi \in \Gamma(f)$ then there exists $\psi \in \Gamma(f)$ which is a $\pm$ directional lift of $\phi$.
In particular, $\Gamma(f)$ satisfies Condition $\Gamma$.\end{prop}

\begin{proof} $\Gamma(f)$ is closed under factors by Corollary \ref{cor2.08} (a).

If $\phi \in \Gamma(f)$ then there exists a partition $\al : X \tto I$ such that $\phi = f^{\al}$.
Construct a basic sequence of partitions $\{ \al_n \}$ with $\al_1 = \al$. By Theorem \ref{theo3.07}
 $\{ (I_n,f^{\al_n}) \}$ is a Shimomura sequence and it is invertible if $f \in H(X)$.  Hence, for sufficiently large
 $n$, the map $p_{n,1} : I_n \to I$ induces a $+$ directional lift $f^{\al_n}$ of $f^{\al_1} \ = \ \phi$ and if
 $f \in H(X)$ then $n$ large enough implies that the lift is $\pm$ directional.

 For general $f \in C_s(X)$, we apply the $\pm$ result to the natural lift $\hat f \in H(\hat X)$. By Corollary \ref{cor2.08} (c)
 $\Gamma(f) = \Gamma(\hat f)$ and so every $\phi \in \Gamma(f)$ admits a $\pm$ directional lift in $\Gamma(f)$.
 \end{proof}\vspace{.5cm}

 Now we obtain our main characterization result.

\begin{theo}\label{theo3.11} Let $K_0 \subset K \subset C_s(X)$ be nonempty. Let $\QQ$ be
a collection of surjective relations on elements
of $\I$.
\begin{itemize}
\item[(a)] $\Gamma(\Gamma^*(\QQ)) \subset \QQ, \ K \subset \Gamma^*(\Gamma(K))$ and
$$\Gamma^*(\Gamma(\Gamma^*(\QQ))) \ = \ \Gamma^*(\QQ), \ \Gamma(K) \ = \ \Gamma(\Gamma^*(\Gamma(K))).\hspace{1cm}$$

\item[(b)] $\Gamma^*(\QQ) \subset C_s(X)$ is closed and conjugacy invariant.

\item[(c)] $\Gamma(K)$ satisfies condition $\Gamma$.

\item[(d)] If $K_0$ is dense in $K$, then $\Gamma(K_0) = \Gamma(K)$.

\item[(e)] If $K_0$ and $K$ are conjugacy invariant and $\Gamma(K_0) = \Gamma(K)$ then
$K_0$ is dense in $K$.

\item[(f)] If $K$ is closed in $C_s(X)$ and is conjugacy invariant then $K = \Gamma^*(\Gamma(K))$.
If $K$ is closed in $H(X)$ and is conjugacy invariant then $K = H(X) \cap \Gamma^*(\Gamma(K))$.

\item[(g)] If $\QQ$ satisfies condition $\Gamma$ then
$$\QQ \ = \ \Gamma (\Gamma^*(\QQ)) \ = \  \Gamma (H(X) \cap \Gamma^*(\QQ)).$$
\end{itemize}
\end{theo}

\begin{proof} (a) The two inclusions are obvious and then imply the equations by monotonicity of $\Gamma$ and $\Gamma^*$.

(b) If $\al : X \tto I$ is a partition and $\{ f_n \}$ is a sequence in $C_s(x)$  converging
to $f$ then $ \al  f_n$ is eventually constant at $\al  f$ and so eventually $f_n^{\al}$
is eventually constant at $f^{\al}$. If $f_n \in \Gamma^*(\mathcal{Q})$ for all $n$ then
$f^{\al} \in \mathcal{Q}$. As $\al$ was arbitrary $f \in \Gamma^*(\mathcal{Q})$. $\Gamma^*(\mathcal{Q})$
is conjugacy invariant by (\ref{2.04}).

(c) Clearly, $\Gamma(K) = \bigcup_{f \in K} \ \Gamma(f)$ and  Condition $\Gamma$ is preserved by arbitrary unions. So
the result follows from Proposition \ref{prop3.10}.

(d) Clearly, $\Gamma(K_0) \subset \Gamma(K)$.  If $\phi \in  \Gamma(K)$ then $f^{\al} = (\al \otimes \al f)(X)$
for some $f \in K$ and $\al : X \tto I$. Because $K_0$ is dense in $K$ there exists $f_0 \in K_0$ such that
 $\al f_0 = \al f$.  Hence, $f_0^{\al} = f^{\al} = \phi$ and so $\phi \in \Gamma(K_0)$.

 (e) If $f \in K$ and $\al : X \tto I$ is a partition then $f^{\al} \in \Gamma(K) = \Gamma(K_0)$ and so there
 exists $\bt : X \tto I$ and $f_0 \in K_0$ such that $f_0^{\bt } = f^{\al}$. By Proposition \ref{prop2.07a} (b) there exists
 a homeomorphism $h \in H(X)$ such that $f \sim_{\al} h^{-1}f_0 h$. Since $K_0$ is conjugacy invariant $h^{-1}f_0 h \in K_0$.
 Since $\al$ was arbitrary, $K_0$ is dense in $K$.

 (f) $K \subset \Gamma^*(\Gamma(K)) = K_1$.  By (b) $K_1$ is closed and conjugacy invariant. By (a) $\Gamma(K) = \Gamma(K_1)$.
 By (d), $K$ is dense in $K_1$.  Since $K$ is closed, it equals $K_1$.

 Now assume that $K \subset H(X)$ is closed in the relative topology.  So if $\ol{K}$ is the closure in $C_s(X)$ then
 $K$ is dense in $\ol{K}$ and $K = \ol{K} \cap H(X)$. Since $\ol{K}$ is closed and conjugacy invariant it follows
 that $\ol{K} = \Gamma^*(\Gamma(\ol{K}))$ which equals $\Gamma^*(\Gamma(K))$ by (c). Hence, $K = \ol{K} \cap H(X) =
 \Gamma^*(\Gamma(K)) \cap H(X)$.

 (g) Let $\phi \in \QQ$.  By Condition $\Gamma$ we can inductively build a Shimomura sequence
 $\{ \ p_{n+1,n} : (I_{n+1},\phi_{n+1}) \tto (I_n,\phi_n) \ \}$ with $\phi_1 = \phi$ and $\phi_n \in \QQ$ for all
 $n \in \N$.  Let $(X,f)$ be the inverse limit of the sequence.  By Theorem \ref{theo3.04} $X$ is a Cantor set with
 $f \in C_s(X)$ and every element of $\Gamma(f)$ is a factor of some $\phi_n$.  Because $\QQ$ is closed under factors,
 we have $\Gamma(f) \subset \QQ$, i. e. $ f \in \Gamma^*(\QQ)$. If $(\hat X, \hat f)$ is the natural lift of $f$ to a
 homeomorphism then by Corollary \ref{cor2.08}(c) $\Gamma(\hat f) = \Gamma(f)$ and so $\hat f \in \Gamma^*(\QQ)$.
 After identification of $\hat X$ with $X$ via a homeomorphism, we have that $\hat f \in H(X) \cap \Gamma^*(\QQ)$ with
 $\phi \in \Gamma(\hat f)$.

 The reverse inclusions follow from (a) and monotonicity of $\Gamma(\cdot)$.
 \end{proof}\vspace{.5cm}

 Thus, $K \mapsto \Gamma(K)$ and $\QQ \mapsto \Gamma^*(\QQ)$ are inverse bijections between the collection of closed,
 conjugacy invariant subsets of $C_s(X)$ and the collection of those sets of surjective relations on elements of $\I$
 which satisfy condition $\Gamma$.

 \begin{cor}\label{cor3.12} If $K \subset C_s(X)$ is closed and conjugacy invariant then
 $H(X) \cap K$ is a dense $G_{\dl}$ subset of $K$. \end{cor}

 \begin{proof} $\QQ = \Gamma(K)$ satisfies Condition $\Gamma$ by (c) and by (f) $K = \Gamma^*(\QQ)$. Hence, (g) implies
 that $\Gamma(K) = \QQ = \Gamma(H(X) \cap K)$.  So by (e) $H(X) \cap K$ is dense in $K$. Since $H(X)$ is a $G_{\dl}$
 subset of $C_s(X)$, its intersection with $K$ is a $G_{\dl}$ subset of $K$.
 \end{proof}\vspace{.5cm}

\begin{theo}\label{theo3.13} If $K,K_0$ are conjugacy invariant subsets of $C_s(X)$ and $K$ is closed, then
$K_0 \subset K$ iff $\Gamma(K_0) \subset \Gamma(K)$.  In particular, for $f \in C_s(X)$ the
closure of the conjugacy class of $f$ is
$\{ g \in C_s(X) : \Gamma(g) \subset \Gamma(f) \}$.
\end{theo}

\begin{proof} Clearly, $K_0 \subset K$ implies $\Gamma(K_0) \subset \Gamma(K)$. By Theorem \ref{theo3.11} (f)
$K = \Gamma^*(\Gamma(K))$ and by (d) and (f) the closure $\overline{K_0} = \Gamma^*(\Gamma(K_0))$. Hence,
$\Gamma(K_0) \subset \Gamma(K)$ implies $\overline{K_0} \subset K$.

If $K$ is the closure of the conjugacy class of $f$ and $K_0$ is the conjugacy class of $g$ then $g \in K$ iff
$K_0 \subset K$ and so iff $\Gamma(g) = \Gamma(K_0) \subset \Gamma(K) = \Gamma(f)$.
\end{proof}\vspace{.5cm}

From Corollary \ref{cor2.08} we  obtain the following result from \cite{S1}.

\begin{cor}\label{cor3.14}(a) If $f,f_1 \in C_s(X)$ and  $(X,f_1)$ is a factor of $(X,f)$
then $f_1$ is in the closure of the
conjugacy class of $f$.

(b) If $f \in C_s(X)$ and $(\hat X, \hat f)$ is the natural lift of $f$ to a homeomorphism then, (after identifying
$\hat X$ with $ X$ ) the closure of the conjugacy classes of $f$ and of $\hat f$ agree. \end{cor}

$\Box$ \vspace{1cm}

Discussed by most of the authors whose work I am describing is the important question of when the conjugacy class of a
map $f$ is a $G_{\dl}$ in $C_s(X)$.  In that case, the class is a residual subset of its closure. Here we will describe
sufficient conditions which are convenient to apply. They come from Shimomura \cite{S3} and from \cite{AGW}.

We will say that a Shimomura Sequence  $\{ \ p_{n+1,n} : (I_{n+1},\phi_{n+1}) \tto (I_n,\phi_n) \ \}$  is
\emph{pointed}\index{Shimomura Sequence!pointed} \index{pointed Shimomura Sequence} 
when $I_1$ is a singleton and so $(I_1,\phi_1)$ is the trivial system. Any Shimomura sequence can be adjusted
to become pointed by inserting the trivial system at level $1$ and shifting the other index numbers up by one. We will call this
the \emph{pointed extension} of the original sequence.

\begin{df}\label{df3.15} Let  $\{ \ p_{n+1,n} : (I_{n+1},\phi_{n+1}) \tto (I_n,\phi_n) \ \}$ be a pointed Shimomura sequence.
If $f \in C_s(X)$ then $f$ satisfies \emph{the lifting property} with respect to the sequence if
for every $\ep > 0, n \in \N$ and $\al : (X,f) \tto (I_n,\phi_n)$ there exists $m > n$ and $\bt : (X,f) \tto (I_m,\phi_m)$
such that $\al = p_{m,n} \circ \bt $ and with $mesh \ \bt < \ep$.\end{df}
  \vspace{.5cm}

 \begin{theo}\label{theo3.16} Let $\{ \ p_{n+1,n} : (I_{n+1},\phi_{n+1}) \tto (I_n,\phi_n) \ \}$ be a pointed Shimomura
 sequence. The set of $f \in C_s(X)$ which satisfy the lifting property with respect to the sequence is a $G_{\dl}$ subset
 of $C_s(X)$. If the set is nonempty then it is exactly the set of $f \in C_s(X)$ which are conjugate to the inverse
 limit of $\{ \ p_{n+1,n} \times p_{n+1,n} : \phi_{n+1} \tto \phi_n \ \}$.\end{theo}

 \begin{proof} Given $\ep > 0$ and $\al : X \tto I_n$ we let $G(\al,\ep)$ be the set of $f \in C_s(X)$ such that
 $f^{\al} = \phi_n$ and there exists $\bt : X \tto I_m$ with $mesh \  \bt < \ep$ and such that $f^{\bt } = \phi_m$ and
 $p_{m,n} \circ \bt = \al$. $G(\al,\ep) \subset C_s(X)$ is open, because the map $\Gamma_0$ is locally constant. That is, for any
 $\al, \bt $ the relations  $f^{\al}, f^{\bt }$ are unaltered as $f$ varies in a small enough open neighborhood.
 The intersection over the countable
 set of $\al$'s and rational $\ep > 0$ is the $G_{\dl}$ set of maps with the lifting property.

 If $f$ satisfies the lifting property then using (\ref{2.04}) it is easy to check that any conjugate of $f$ satisfies the
 lifting property.

 Now assume that $f$ satisfies the lifting property. Since $(I_1,\phi_1)$ is trivial, there is a unique
 $\al_1 : (X,f) \tto (I_1,\phi_1)$. Let $n_1 = 1$.  Inductively, we use the lifting property to
 define an increasing sequence $\{ n_i \}$
 and $\al_i : (X,f) \tto (I_{n_i},\phi_{n_i})$ with $mesh \  \al_i < 1/i$ for $i > 1$ and
 with $\al_i = p_{n_{i+1},n_i} \circ \al_{i+1}$.

 Now define $q_n : (X,f) \tto (I_n,\phi_n)$ for all $n \in \N$ by $q_n = \al_{i}$ if $n = n_i$, and
 by $q_n = p_{n_{i+1},n} \circ \al_{i+1}$ if $n_i \leq n < n_{i+1}$. These define a map from $(X,f)$ onto
 the inverse limit $(I_{\infty},\phi_{\infty})$. Since the $mesh \  \al_i  \ \to \ 0$ as $i \ \to \infty$ it follows that
 the map from $X$ to $I_{\infty}$ is injective and so is a homeomorphism.  This shows that $f$ is conjugate to the
 inverse limit map.
 \end{proof}\vspace{.5cm}

 We show that for homeomorphisms, the lifting property is preserved by discrete suspension. This requires a little construction.

 \begin{lem}\label{lem3.16a} Assume  $f \in H(X)$, $R$ is a closed relation on $A$ and $N$ is a positive integer.
 If $p : (X_N,f_N) \to (A_N,R_N)$ is a continuous map between the $N$-fold suspensions, then there exist a
 homeomorphism $h : (X_N,f_N) \to (X_N,f_N)$ and a continuous map $\tilde p : (X,f) \to (A,R)$ such that
 $p \circ h = (\tilde p)_N$. \end{lem}

 \begin{proof} Let $p(x,k) = (p_1(x,k),n(x,k))$ for $(x,k) \in X \times [1,N]$. Since $p$ maps $f_N$ to $R_N$, it
 maps $f_N^N$ to $R_N^N$.  By (\ref{1.03}) this says that $p(f(x),k) = (p_1(f(x),k),n(f(x),k)) $ has
 $n(f(x),k) = n(x,k)$ and $(p_1(x,k),p_1(f(x),k)) \in R$.

 Since $n : X_N \to [1,N]$ is continuous, it follows that $X_j = \{ x \in X : n(x,1) = j \}$ is a clopen, $f$ invariant subset of
 $X$ for $j = 1,\dots,N$.  Define $h$ to equal $(f_N)^{1-j}$ on the clopen invariant $ X_j \times [1,N] \subset X_N$. Since
 $X_N$ is thus decomposed into invariant pieces, $h$ maps $(X_N,f_N)$ to itself. Since $n(f_N(x,k)) = n(x,k) + 1 \ mod \ N$,
 we have that $n(h(x,1)) = 1$ for all $x$.  Thus, if we let $\hat p(x) = p_1(h(x,1))$ we see that $p(h(x,1)) = (\tilde p(x),1)$
 and so for $k = 1, \dots,N, \ p(h(x,k)) = (\tilde p(x),k)$.  That is, $p \circ h = (\tilde p)_N$.

 If $x \in X_j$ then $h(x,1) = (x,1)$ if $j = 1$ and $\tilde p(f(x)) = p_1(f(x),1)$. Otherwise $h(x,1) = (f^{-1}(x),N-j+2)$,
 $\tilde p(x),1) = p_1(f^{-1}(x),N-j+2)$
 $\tilde p(f(x)) = p_1(h((f_N)^N(x,1)) = p_1((f_N)^N(h(x,1)) = p_1(x,N-j+2)$.
 In either case, we have $(\tilde p(x), \tilde p(f(x))) \in  R$.  That is, $\tilde p : (X,f) \to (A,R)$.
 \end{proof}\vspace{.5cm}

 \begin{theo}\label{theo3.16b}  If $f \in H(X)$ satisfies
 the lifting property for the pointed Shimomura sequence $\{ \ p_{n+1,n} : (I_{n+1},\phi_{n+1}) \tto (I_n,\phi_n) \ \}$
 and $N$ is a positive integer then the $N$-fold suspension $f_N$ on $X_N$ satisfies the lifting property
 for the pointed extension of  $\{ \ (p_{n+1,n})_N : ((I_{n+1})_N,(\phi_{n+1})_N) \tto ((I_n)_N,(\phi_n)_N) \ \}$.
 \end{theo}

 \begin{proof} If $\al : (X,f) \tto (I_n,\phi_n)$ then $(\al)_N : (X_N,f_N) \tto ((I_n)_N,(\phi_n)_N)$ is a
 lift of the map to the trivial system. That is, we can lift the bottom level of the pointed extension of the suspension.

 Given $\al : (X_N,f_N) \tto ((I_n)_N,(\phi_n)_N)$ and $\ep > 0$, apply Lemma \ref{lem3.16a} to get $h : (X_N,f_N) \to (X_N,f_N)$
 a homeomorphism and $\gamma : (X,f) \tto (I_n,\phi_n)$ so that $\al \circ h = (\gamma)_N$. Let $\dl > 0$ be an
 $\ep$ modulus of uniform continuity for $h$ and choose  $\bt : (X,f) \tto (I_{n+k},\phi_{n+k})$ so that
 $p_{n+k,n} \circ \bt = \gamma$ and the mesh of $\bt $ is less than $\dl$. Then the mesh of $(\bt )_N$ is less than
 $\dl$ and so $\tilde \bt = (\b)_N \circ h^{-1}$ has mesh less than $\ep$.   Furthermore,
 $$(p_{n+k,n})_N \circ \tilde \bt \ = \ (\gamma)_N \circ h^{-1} \ = \ \al,$$
 as required.
 \end{proof}\vspace{.5cm}

 We describe an alternative to the lifting property which is a bit easier to use.

 \begin{df}\label{df3.17}    If $\{ \ p_{n+1,n} : (I_{n+1},\phi_{n+1}) \tto (I_n,\phi_n) \ \}$ is a pointed Shimomura sequence,
 then we says that the
 the sequence has the \emph{factoring  property}\index{factoring  property} if whenever $q_1 : (I_k,\phi_k) \tto (I_n,\phi_n)$ with $k > n$, there
 exists  $m > k$ and $q_2 : (I_m,\phi_m) \tto (I_k,\phi_k)$  such that $q_1 \circ q_2 = p_{m,n}$.
 \end{df}
\vspace{.5cm}

\begin{theo}\label{theo3.18} Let $\{ \ p_{n+1,n} : (I_{n+1},\phi_{n+1}) \tto (I_n,\phi_n) \ \}$ be a pointed Shimomura
sequence with limit $(X,f)$.
 The sequence has the factoring property iff $f$ has the lifting property with respect to the sequence.
\end{theo}

\begin{proof}  Write $\al_n = p_n: (X,f) \tto (I_n,\phi_n)$, the projection from the inverse limit.

  Assume $f$ has the lifting property and we are given $q_1 : (I_k,\phi_k) \tto (I_n,\phi_n)$.

Let $\ep > 0$ be
smaller than the Lebesgue number of $\al_k$. By the lifting property applied to $q_1 \circ \al_k$,
there exists $q : (X,f) \tto (I_m,\phi_m)$
such that $p_{m,n} \circ q = q_1 \circ \al_k$ and with $mesh \  q < \ep$. Because $\ep$ is smaller than the
Lebesgue number of $\al_k$, $q_2 = \al_k \circ q^{-1} : I_m \tto I_k$ is a map with $q_1 \circ q_2 = p_{m,n}$.
We can always replace $q_2$ by $q_2 \circ p_{m', m}$ with $m' > m$ arbitrarily large and so we can get $m > k$ if it
was not already.  This is the factoring property.

Now assume the factoring property and suppose we are given $\ep > 0, n \in \N$ and $\al : (X,f) \tto (I_n,\phi_n)$.

Because $(X,f)$ is the inverse limit, $\al$ factors through $\al_k : X \tto I_k$ for $k$ sufficiently large.  This is because
any clopen set is a finite union of basic sets and $\A^{\al}$ consists of finitely many clopen sets.
We choose $k$ large enough that $k > n$ and $mesh \  \al_k < \ep$. Thus, there exists $q_1 : (I_k,\phi_k) \tto (I_n,\phi_n)$
so that $\al = q_1 \circ \al_k$. By the factoring property, there exist $m > k$ and $q_{2} : (I_m,\phi_m) \tto (I_k,\phi_k)$
so that $q_1 \circ q_{2} = p_{m,n}$.

Now use the factoring property again to get $q_{3} : (I_r,\phi_r) \tto (I_m,\phi_m)$ so that
$q_{2} \circ q_{3} = p_{r,k}$. Let $\bt = q_3 \circ \al_r : (X,f) \tto (I_m,\phi_m)$.  Since, $q_2 \circ \bt =
p_{r,k} \circ \al_r = \al_k$ it
follows that $mesh \  \bt \ \leq \ mesh \  \al_k \ < \ \ep$. Furthermore,
$$p_{m,n} \circ \bt \ = \ q_1 \circ q_2 \circ q_3 \circ a_r \ = \ q_1 \circ p_{r,k} \circ a_r \ = \ q_1 \circ \al_k \ = \ \al.$$
This proves the lifting property.
\end{proof}\vspace{1cm}

We will say that $f \in C_s(X)$ is \emph{of residual type} when its conjugacy class is a $G_{\dl}$ subset of $C_s(X)$ and
so is a residual subset in its closure.

\begin{theo}\label{theo3.19} If $f \in C_s(X)$ is of residual type then $f \in H(X)$. \end{theo}

\begin{proof} Let $K$ be the closure in $C_s(X)$ of the conjugacy class of $f$.  Hence, $K$ is a
closed, conjugacy invariant subset of $C_s(X)$.  By Corollary \ref{cor3.12} $H(X) \cap K$ is a dense, $G_{\dl}$
subset of $K$. By definition, the conjugacy class of $f$ is dense in $K$. If the conjugacy class is a $G_{\dl}$
then these two residual subsets intersect.  That is, there exists $g \in H(X)$ which is conjugate to $f$. Since $H(X)$
is itself conjugacy invariant, $f \in H(X)$.
\end{proof}\vspace{.5cm}

We will use the lifting property and the factoring property to construct examples of residual type. However, I do
not know whether the property is necessary, i.e. whether any $f$ of residual type admits a Shimomura sequence with respect to
which it has the lifting property.  So, for example, I do not know if the discrete suspensions of a homeomorphism of
residual type are necessarily of residual type. Nonetheless, Theorem \ref{theo3.16b} will suffice for construction purposes.
\vspace{1cm}

\section{Examples}\label{examples}
\vspace{.5cm}

Following Bernardes and Darji \cite{BD}, we define for  integers
$N,  M \geq 1, L \geq 0 $, an $N$\emph{loop}, or a loop of length $N$,  to be a relation isomorphic
to $\phi_N$ on $[1,N] = \{ 1,..., N \} \in \I$  given by $\{ (i,i+1) : i = 1,...,N-1 \} \cup \{ (N,1) \}$,
i.e.  translation by $1$ on the the group $\Z/N \Z$.
An $N-L-M$ \emph{dumbbell}\index{dumbbell} is a relation isomorphic to $\phi_{N,L,M} $ on
$[1,N+L+M-1] = \{ 1,..., N + L + M - 1 \} \in \I$  given by
$\{ (i,i+1) : i = 1,...,N + L + M-2 \} \cup \{ (N,1), (N + L + M - 1, N + L) \}$.
Notice that a $1$ loop and a $1-0-1$ dumbbell are both the trivial surjective relation on a singleton.

For the $N-L-M$ dumbbell, $\phi_{N,L,M}$ restricts to a loop on $[1,N]$. It is called the \emph{in-loop}\index{in-loop} of the dumbbell.
The restriction to $[N+L,N+L+M-1]$ is a loop called the \emph{out-loop}\index{out-loop} of the dumbbell. We call the restriction to
$[N,N+L]$ the \emph{connecting path}\index{connecting path} of the dumbbell. We extend the language via the
-unique- isomorphism to any $N-L-M$ dumbbell.

If $L \geq 1$ then the inloop points precede the outloop points with respect to the partial order given by $\OO \phi_{N,L,M}$.
On the other hand, if $L = 0$, then the connecting path is trivial and we call the $N-0-M$ dumbell an $N-M$  \emph{wedge}\index{wedge}. In that
case, $\phi_{N,0,M}$ is isomorphic to $\phi_{M,0,N}$ by the map which sends $i$ to $M+i$ for $i = 1,.., N-1$ and to
$i - N + 1$ for $i = N ,..., N+M-1$.

If $L \geq 1$ and $K$ is a positive integer, then the $K$-fold suspension
$(\phi_{N,L,M})_K$ is isomorphic to $\phi_{NK,(L - 1)K + 1,MK}$ by
$(i,j) \mapsto K(i - 1) + j$.  If $L = 0, K > 1$, then $ (\phi_{N,0,M})_K$ is not a dumbell as the two loops share a common path
of length $K$

\begin{prop}\label{prop4.01} Every surjective relation $\phi$ on a finite set $I$ is a factor of a
finite disjoint union of dumbbells.

Every transitive $\phi$ on $I$
is a factor of a single loop.

Every recurrent $\phi$ on $I$ is a factor of a finite disjoint union of loops which may taken to be of the same length.
\end{prop}

\begin{proof} We can regard a surjective relation $\phi \subset I \times I$ as describing a directed graph
with vertices $I$ and with edges given by the pairs in $\phi$.

If $\phi$ is transitive, i.e. $I$ is a single basic set
then we can choose a path along the graph which begins and ends at the same vertex and which passes through every edge.
This expresses $\phi$ as a factor of a loop.  If $I$ is recurrent and so is the union of basic sets then $\phi$ is the
factor of a finite number of loops. Taking the least common multiple $N$ of the lengths, we can lift each loop to an
$N$ loop and so express $\phi$ as a factor of a disjoint union of $N$ loops.

We can include any edge  for a surjective relation in a path and extend it forward and backward until on each side a
repeat of a vertex occurs. This exhibits a dumbbell which maps into $\phi$ and which hits the given edge.  As there
are only finitely many edges we can express $\phi$ as a factor of a finite union of dumbbells.
\end{proof}\vspace{.5cm}

\begin{lem}\label{lem4.02} If $f \in H(X)$ and $\phi_N$ is the $N$-loop on $[1,N]$ then the following are equivalent:
\begin{itemize}
\item[(i)]  There exists $p : (X,f) \to ([1,N],\phi_N)$.
\item[(ii)] There exists a partition $\al : X \tto [1,N]$ such that $f^{\al} = \phi_N$.
\item[(iii)] $\phi_N \in \Gamma(f)$.
\item[(iv)]  $(X,f)$ is an $N$-fold suspension of some system $(Y,g)$.
\end{itemize}
If $p : (X,f) \to ([1,N],\phi_N)$ exists, it is necessarily a system surjection. \end{lem}

\begin{proof} If $p : (X,f) \to ([1,N],\phi_N)$ then $p(f^k(x))$ is congruent to $p(x) + k \  mod \ N$. Thus, $p(X) = [1,N]$
and so $p : (X,f) \tto ([1,N],\phi_N)$ since $\phi_N$ is a map (see Lemma \ref{lem1.01a}(b).
Since such a system map $p$ is necessarily a surjection and $\phi_N$ is a
permutation, the equivalence of (i), (ii) and (iii) follows from Proposition \ref{prop2.09}.

(iv) $\Rightarrow$ (i): If $(X,f) = (Y_N,g_N)$ and so $X = Y \times [1,N$ then the second coordinate projection from $X$ to $[1,N]$
maps $X$ onto $[1,N]$ and $f$ onto $\phi_N$.

(i) $\Rightarrow$ (iv): If $p : (X,f) \to ([1,N],\phi_N)$ then we let $Y = p^{-1}(1)$
which is a clopen $g = f^N$ invariant subset of $X$. Define the homeomorphism $h : Y_N \to X$ by
$h(x,i) = f^{i-1}(x)$ for $i = 1,...,N$. Clearly, for $i < N$, $f(h(x,i)) = f^i(x) = h(x,i+1) =h(g_N(x,i))$, while
$f(h(x,N)) = f^N(x) = h(f^N(x),1)) = h(g_N(x,N))$.
\end{proof}\vspace{.5cm}

We will repeatedly use the rigidity of maps between loops and dumbbells as described in the following Lemma.

\begin{lem}\label{lem4.03} (a) There exists a map $p : ([1,M],\phi_M) \to ([1,K],\phi_K)$ iff $K$ is a divisor of $M$.
In that case, given $i \in [1,M]$ and $j \in [1,K]$ there is a unique such $p$, necessarily surjective, with
$p(i) = j$.

(b) There exists a map $p : ([1,N + L + M - 1],\phi_{N,L,M}) \to ([1,K],\phi_K)$ iff $K$ is a divisor of $N$ and of $M$.
In that case, given $i \in [1,N + L + M - 1]$ and $j \in [1,K]$ there is a unique such $p$, necessarily surjective, with
$p(i) = j$.

(c) If $p : ([1,N + L + M - 1],\phi_{N,L,M}) \to ([1,N_1 + L_1 + M_1],\phi_{N_1,L_1,M_1})$ and $p$ is not surjective,
 then either $p([1,N + L + M - 1])$ is the in-loop or the out-loop of $\phi_{N_1,L_1,M_1}$.

(d) Assume $L_1 \geq 1$. There exists a map
$p : ([1,N + L + M - 1],\phi_{N,L,M}) \tto ([1,N_1 + L_1 + M_1],\phi_{N_1,L_1,M_1})$ iff
$N_1$ is a divisor of $N$, $M_1$ is a divisor of $M$ and $L_1 \leq L$. If $N \leq i \leq N + L$ and
$N_1 \leq j \leq N_1 + L_1$ then such a $p$ exists with $p(i) = j$ iff $i - N \geq j - N_1$ and $N + L - i \geq N_1 + L_1 - j$.
In that case, $p$ is unique.

 \end{lem}

 \begin{proof} (a): $\phi_M$ is a map with $(\phi_M)^M = 1_{[1,M]}$ and so if $p$ maps $\phi_M$ to $\phi_K$ then
 $(\phi_K)^M = 1_{[1,K]}$ and so $K|M$. Since $[1,M]$ is the $\phi_M$ orbit of any of its points, the map $p$ is uniquely
 determined by its value on any point. By rotating in the $\phi_K$ loop we see that the value of $p(i)$ can be
 arbitrarily chosen.

 (b): By (a) applied to the restriction of $p$ to the inloop and the outloop, we see from (a) that for $p$ to exist we must
 have $K|N$ and $K|M$. In that case, the map from $[1,N + L + M - 1]$ to $[1,K]$ obtained by sending $i$ to its congruence class
 $mod \ K$ is a map from $\phi_{N,L,M}$ to $\phi_K$. Again we can rotate the $\phi_K$ loop to get an arbitrary value
 of $p(i)$.

  If $N \leq i \leq N+L$ then by moving forward and backward we see that
 $p$ is uniquely determined by the value $p(i)$. If $i$ is on the inloop then (a) implies that
 the restriction to the inloop is uniquely determined by $p(i)$.  Thus, $p(N)$ is determined and since $N$ is on the connecting
 path, $p$ is determined.  Similarly, if $i$ is on the outloop.

 (c): The only proper subsystems of the dumbbell are the inloop, the outloop  and their disjoint union.
 The image of $p$ is a subsystem
 in which any two elements can be connected by a chain.
 It follows that the image is either the inloop our the outloop when $p$ is not onto.
 When $p([1,N + L + M - 1]) = [1,N_1 + L_1 + M_1 - 1]$ then
 the map of systems is surjective by Lemma \ref{lem1.01a}(b).

 (d): When $L_1 > 0$, the dumbbell $\phi_{N_1,L_1,M_1}$ is not transitive.  The only transitive subsets are the inloop and the
 outloop.  The image of the inloop and the outloop of $\phi_{N,L,M}$ are transitive subsets. Furthermore, if $L = 0$ then
 $\phi_{N,L,M}$ would be transitive and with a transitive image.  Hence, $L > 0$.  Since $p$ maps
 $\OO \phi_{N,L,M}$ to $\OO \phi_{N_1,L_1,M_1}$ we see that $p$ takes inloop to inloop and outloop to outloop. Note that
 this also uses $L_1, L > 0$.  It thus follows from (a) that $N_1|N$ and $M_1|M$ are required.

 If $N_1 < j < N_1 + L_1$ then $j$ is not in one of the endloops and so any pre-image is not in an endloop.
 Also, $N_1 + L_1 - 1$ is not in the outloop and $N_1 + L_1 \in phi_{N_1,L_1,M_1}(N_1 + L_1 -1)$. Similarly,
 $N_1 \in \phi_{N_1,L_1,M_1}^{-1}(N_1 + 1)$. Thus, the preimage of the points
 of the connecting path in $\phi_{N_1,L_1,M_1}$ must lie in the connecting path of $\phi_{N,L,M}$.  Hence, $L \geq L_1$.
Furthermore, if $N_1 \leq j \leq N_1 + L_1$, i.e. $j$ lies on the connecting path in the image, and $p(i) = j$ then
$p(i-(j - N_1)) = N_1$ and $p(i + (N_1 + L_1 - j)) = N_1 + L_1$.  Hence, $i - (j - N_1)$ and $i + (N_1 + L_1 - j)$
lie in the connecting path of $\phi_{N,L,M}$. Thus, $N \leq i - (j - N_1)$ and $i + (N_1 + L_1 - j) \leq N + L$.

Finally, given $i$ and $j$ satisfying the above conditions, $p$ is uniquely defined by taking $i + k$ to $j + k$
for $N_1 - j \leq k \leq N_1 + L_1 - j$. At the endpoints we are in the inloop on the left and the outloop on the right.
We then move around the loops.
\end{proof}\vspace{.5cm}

{\bfseries Remark:}  Notice that when $N = N_1, L = L_1, M = M_1$ and $L_1 \geq 1$ then the only surjection between
the dumbbells has $i = j$.  That is, the identity map is the unique surjection from a dumbbell to itself when the connecting
path is nontrivial.
\vspace{.5cm}

We will call a closed, conjugacy invariant subset $K$ of $C_s(X)$ \emph{conjugacy transitive}\index{conjugacy transitive} when
it is the closure of the conjugacy class of
some $f \in K$, i.e. $K = \overline{H(X)\cdot f}$.  Such elements $f$ are called
\emph{conjugacy transitive elements} and  the set of
such elements is denoted $Trans(K)$. Thus, $K$ is  conjugacy transitive when $Trans(K)$ is nonempty,
in which case it is a dense $G_{\dl}$ subset of $K$. We call $K$ \emph{conjugacy minimal} when $Trans(K) = K$. By Theorem
\ref{theo3.11}(e)
\begin{equation}\label{4.01aa}
 f \in Trans(K) \qquad \Leftrightarrow \qquad \Gamma(f) \ = \ \Gamma(K).
 \end{equation}

A conjugacy transitive point
$f$ is of residual type when the conjugacy class $H(X) \cdot f$ is a $G_{\dl}$ subset of $K$. Since distinct conjugacy classes
are disjoint and two dense $G_{\dl}$ subsets meet, it follows that there is at most one dense $G_{\dl}$ conjugacy class. Theorem
\ref{theo3.19} implies that such residual type transitive points are contained in $H(X)$.

If $\QQ$ is a collection of surjective relations on members of $\I$ and $K \subset C_s(X)$ we will say that $\QQ$
\emph{generates}\index{generates} $\Gamma(K)$ if $\QQ \subset \Gamma(K)$ and every element of $\Gamma(K)$ is a factor of some relation in
$\QQ$.
\vspace{1cm}

{\bfseries Example 1 - $\mathbf{ C_s(X)}$:}
\vspace{.5cm}

By Theorem \ref{theo3.01}   $\Gamma(C_s(X))$ is the set of all surjective relations on members of $\I$.
Since the set of surjective relations on elements of $\I$ is countable,  Corollary \ref{cor3.02a} and the Baire Category theorem
imply that the set of $f \in C_s(X)$ with $\Gamma(f) = \Gamma(C_s(X))$ is a dense $G_{\dl}$
subset of $C_s(X)$. Thus, $C_s(X)$ is conjugacy transitive and this set is $Trans(C_s(X))$ by (\ref{4.01aa}).

Let $f, g \in C_s(X)$ with $g$ a conjugacy
transitive point of $C_s(X)$. If $ (X,g)$ is a factor of $(X,f)$ then $f$ is a conjugacy transitive point of
$C_s(X)$ because
$$\Gamma(C_s(X)) \ = \ \Gamma(g) \ \subset \ \Gamma(f) \ \subset \ \Gamma(C_s(X)).$$
In particular, if $g \in H(X)$ is a conjugacy transitive homeomorphism and $g_1$ is an arbitrary member of $C_s(X)$
then by using a homeomorphism from $X \times X$ to $X$, we can regard $f = g \times g_1$ as a member of $C_s(X)$
which has $g$ as a factor and so is conjugacy transitive.  On the other hand, if $g_1$ is not injective then
$f$ is not and so its conjugacy class is a subset of $C_s(X) \setminus H(X)$ which is dense in $C_s(X)$.

We describe a Shimomura sequence with the factoring property whose limit is a conjugacy transitive element of $C_s(X)$.
For $n \geq 2, i = 1,...,4^{n-1}$ let $\{ (I_{i;n},\phi_{i;n}) \}$ be a collection of $n!-2^{n+1}-n!$ dumbbells with
the sets  $\{ I_{i;n} \}$ pairwise disjoint. Label by
$a_{i;n}, e_{i;n}, b_{i;n}$ the vertices in $I_{i;n}$ at positions $n!, n! + 2^n, n! + 2^{n+1}$ respectively.  That is,
these are the vertices at the left end, the mid-point, and the right end of the connecting path for $\phi_{i;n}$.
Let $(J_1,\Phi_1)$ be a trivial system and for $n \geq 2$ let $(J_n,\Phi_n)$ be the disjoint union of the systems
$\{ (I_{i;n},\phi_{i;n}) : i = 1, \dots , 4^{n-1} \}$. $P_{2,1}$ is the unique map to the trivial system. For
$n \geq 2,$ $P_{n+1,n}$ on $(I_{i;n+1}, \phi_{i;n+1})$ is the unique map to $(I_{j;n},\phi_{j;n})$ taking
$e_{i;n+1}$ to $e_{j;n}$ when $i = 4(j-1) + 2$ or $i = 4(j-1) + 3$. When $i = 4(j-1) + 1$ it is the unique
map onto the inloop which takes $e_{i;n+1}$ to $a_{j;n}$. When $ i = 4(j-1) + 3$ it is the unique map onto the
outloop taking $e_{i;n+1}$ to $b_{j;n}$. Thus, each level $n$ dumbbell is the image of two level $n+1$ dumbbells
and each endloop at level $n$ is the image of a single level $n+1$ dumbbell. Hence, for $k \geq 1$
each level $n$ dumbbell is the image under $P_{n+k,n}$ of $2^k$ dumbbells and each endloop is the image of
$2^{k-1}$ dumbbells.

\begin{theo}\label{theo4.04} $\{ P_{n+1,n} : (J_{n+1},\Phi_{n+1}) \tto (J_n, \Phi_n) \}$ is an invertible pointed
Shimomura sequence which satisfies the factoring property.
If $(X,f)$ is the limit, then $f \in H(X)$ is a conjugacy transitive point for $C_s(X)$
of residual type. \end{theo}

\begin{proof} The map $P_{n+1,n}$ takes the first $2^n$ vertices of the connecting path for $\phi_{i;n+1}$ into
the same endloop as the initial vertex and the last $2^n$ vertices to the same endloop as the last vertex, it follows
that $P_{n+1,n}$ realizes a $\pm$ directional lift and so the sequence is an invertible Shimomura sequence.

If $(X,f)$ is the limit then $f \in H(X)$ and by Theorem \ref{theo3.04}  $\Gamma(f)$ is generated by the $\{ \Phi_n \}$. Every
finite union of dumbbells is clearly a factor of $(J_n,\Phi_n)$ for sufficiently large $n$ and by Proposition
\ref{prop4.01} every surjective relation is a factor of a finite union of dumbbells.  Hence, $\Gamma(f) = \Gamma(C_s(X))$ and
so $f$ is a conjugacy transitive point for $C_s(X)$.

It remains to check the factoring property which will imply that $f$ is of residual type.

If $p : (J_k, \Phi_k) \tto (J_1,\Phi_1)$ then $P_{k+1,k}$ provides the required factoring since $\Phi_1$ is trivial.
Let $p : (J_k, \Phi_k) \tto (J_n,\Phi_n)$ with $k > n > 1$. For each $i = 1,...,4^{n-1}$ the dumbbell $(I_{i;n},\phi_{i;n})$
is hit by at least one and at most $4^{k-1}$ dumbbells in $\Phi_k$.  In addition, each endloop may be the image of
some dumbbells in $\Phi_k$, again at most  $4^{k-1}$ of them.  Choose $m$ large enough that $2^{m-n-1} > 4^{k-1}$.

First we allocate each the dumbbells of $\Phi_m$ to a dumbbell of $\Phi_k$.
For each $\phi_{i;n}$ we distribute the dumbbells which hit it via $P_{m,n}$ among
those which hit it via $p$ so that each of the latter is allocated at least one.
The dumbbells which are mapped onto the inloop of $\phi_{i;n}$ are distributed among the
dumbbells, if any, which are mapped onto the inloop via $p$ again so that each of the latter receives at least one.
If there are none then these level $m$ dumbbells are allocated to map onto the inloop of some dumbbell of $\Phi_k$
which maps onto $\phi_{i;n}$ via $p$. Similarly, we allocate for the outloops.

Having made these allocations the maps are determined as follows. If $p$ maps $\phi_{r;k}$ onto $\phi_{i;k}$ then
the vertex $e_{i;k}$ is the image of a vertex $e$ in the connecting path for $\phi_{r;k}$. If $\phi_{s;m}$ has been
allocated to $\phi_{r;k}$ then we choose the unique map $(I_{s;m},\phi_{s;m}) \tto (I_{r;k},\phi_{r;k})$ which maps
$e_{s;m}$ to $e$. Because $m > k$ the inequalities given in Lemma \ref{lem4.03} (d) are satisfied.  If $P_{m,n}$
maps $\phi_{s;m}$ onto the inloop of $\phi_{i;k}$ then $\phi_{s;m}$ has been allocated to $\phi_{r;k}$ which either maps
onto $\phi_{i;n}$ or onto its inloop. In either case, we map $\phi_{s;m}$ onto the inloop of $\phi_{r;k}$ in such a way
that $e_{s;m}$ is mapped onto a vertex  of the inloop of $\phi_{r;k}$ which is mapped by $p$ onto $a_{i;k}$. There may be several
of these and any one will do.  We use a similar procedure for the outloops.
\end{proof}\vspace{.5cm}

An explicit description is given in \cite{AGW} of this map $f$, unique up to conjugacy, which is  conjugacy
transitive for $C_s(X)$ and is of residual type. The existence of such an element had earlier been proved in \cite{KR}.

Call  $f \in H(X)$ \emph{simple}\index{homeomorphism!simple} if there are only finitely many  chain recurrent points.
This is equivalent to saying that there are finitely many periodic points and the chain recurrent points
are periodic. Call $f \in C_s(X)$ \emph{very simple}\index{homeomorphism!very simple} if there are only finitely many chain recurrent points
and every point of $X$ is eventually mapped to a periodic point.
We now reprove a result of Batista et al \cite{BGT} showing that the simple homeomorphisms and the very simple
maps are each dense in $C_s(X)$.

Let $T$ be the translation homeomorphism on $\Z$, $T(n) = n+1$. With $C$ a Cantor set let
$\tilde X = \{ e_-,e_+ \} \cup \Z \times C$ be the two-point compactification of $\Z \times C$ so that
$\{ n \} \times C \to e_{\pm}$ as $n \to \pm \infty$. Let $\tilde f \in H(\tilde X)$ be the homeomorphism
which extends $T \times 1_C$ on $\Z \times C$.

\begin{prop}\label{prop4.05}  The homeomorphism $\tilde f$ is of residual type in $C_s(\tilde X)$.  \end{prop}

\begin{proof} We will sketch the proof leaving the final details
to the interested reader.

We use as our Cantor set $C = \{ 1,2 \}^{\N}$, with $C_N = \{ 1,2 \}^N$ the set of words of length $N$. We define an
inverse sequence. Let $(I_1,\phi_1)$ be the trivial system.
For $n \geq 2$, let
\begin{equation}\label{4.01}
\begin{split}
I_n \quad = \quad [1,2^{n+1}] \times C_n \cup \{ (1,1), (2^{n+1},1) \}, \hspace{2cm} \\
\phi_n \quad = \quad \{ ((i,w),(i+1,w) : 1 \leq i < 2^{n+1}, w \in C_n \}  \quad  \cup\\
 \{ ((1,1),(1,w)) : w \in C_n \cup \{ 1 \} \} \ \cup \ \{ ((2^{n+1},w),(2^{n+1},1) : w \in C_n \cup \{ 1 \} \}.
 \end{split}
 \end{equation}
 and
 \begin{equation}\label{4.02}
 p_{n+1,n}(i,wa) \quad = \quad \begin{cases} (1,1) \quad \mbox{for} \ 1 \leq i \leq 2^{n}, wa \in C_{n+1} \cup \{ 1 \} \\
 (i-2^n,w) \quad \mbox{for} \ 2^n < i \leq 2^n + 2^{n+1},  wa \in C_{n+1}, \\
 (2^{n+1},1) \quad \mbox{for} \ 2^n + 2^{n+1} < i \leq 2^{n+2}, wa \in C_{n+1} \cup \{ 1 \}. \end{cases}
 \end{equation}

 Thus, we think of $I_n$ as a set of parallel threads indexed by $C_n$ and gathered at the endpoints  $(1,1)$ and
 $(2^{n+1},1)$.

 Define the map $h_n : \tilde X \to I_n$ by mapping $e_-$ to $(1,1)$, $e_+$ to $(2^{n+1},1)$ and
 \begin{equation}\label{4.03}
 (i,z) \ \mapsto \ \begin{cases} (1,1) \quad \mbox{for} \ i \leq -2^{n}, \\
 (i + 2^{n}, \pi_n(z)) \quad \mbox{for} \ -2^n < i \leq 2^n, \\
 (2^{n+1},1) \quad \mbox{for} \ 2^n < i. \end{cases}
 \end{equation}
 It is easy to check that $\{ h_n \}$ induces an isomorphism from $(\tilde X, \tilde f)$ to the inverse limit.
 Also it is clear that the sequence is an invertible Shimomura sequence.

 Now suppose $p : (I_k,\phi_k) \tto (I_n, \phi_n)$ with $k > n > 1$. The factoring property is obtained from
 the following observations.
 \begin{itemize}
 \item Each endpoint of $\phi_k$ is mapped to an endpoint of $\phi_n$.
 \item If the two endpoints of $\phi_k$ were mapped to the same endpoint of $\phi_n$ then the image, which
 is $\phi_n$ would be a transitive relation which it is not.
 \item The orderings given by $\OO \phi_k$ and $\OO \phi_n$ then imply that the left endpoint is mapped to the left and the
 right to the right.
 \item Each thread of $\phi_k$ is mapped by $p$ onto a thread of $\phi_n$.
 \item If $(i,w)$ is a vertex on a thread in $\phi_n$ and it is the image of the $w'$ thread in $\phi_k$ then
 there is a unique $j$ such that $p(j,w') = (i,w)$ and this equation uniquely determines $p$ on the $w'$ thread.
 \item If $m > k > n$ then each thread in  $I_n$ is the image of $2^{m-n}$ threads in $I_m$.
 \end{itemize}

 It follows that with $m$ large enough we can use the image under $p$ of each $I_k$ thread to allocate
 at least one $I_m$ thread with the same image under $p_{m,n}$. Then we pick a point in each $I_n$ thread and
 pull back to define the map uniquely given the allocations.
 \end{proof}\vspace{.5cm}


If $\al_L : \tilde X \to [1,L]$ by
\begin{equation}\label{4.03a} \al_L(x) \quad = \quad
\begin{cases} 1 \quad \mbox{for} \ x = e_-, \ \mbox{ and} \ x = (i,c) \ \mbox{with} \ i \leq 1, \\
i \quad \mbox{for} \ \ x = (i,c) \ \mbox{with} \ 1 \leq i \leq L, \\
L  \quad \mbox{for} \ x = e_+, \ \mbox{ and} \ x = (i,c) \ \mbox{with} \ i \geq L,\end{cases}
\end{equation}
then $\tilde f^{\al_L} = \phi_{1,L,1}$.
Hence, for every
positive $L$ the $1-L-1$ dumbbell is a factor of $(\tilde X,\tilde f)$.

For $N$ a positive number, we define the factor map $ \pi_N : (\tilde X_N, \tilde f_N) \tto (\tilde X, \tilde f)$ by
\begin{equation}\label{4.04}
\pi_N(x,n) \quad = \quad \begin{cases}  e_{\pm} \quad \mbox{for} \ x = e_{\pm}, \quad 1 \leq n \leq N, \\
(N(i-1) + n,c) \quad \mbox{for} \ x = (i,c), \ i \in \Z, \ \ 1 \leq n \leq N. \end{cases}
\end{equation}
The map is one-to-one except over the endpoints which are lifted to periodic orbits of period $N$. By using the
identification defined by (\ref{1.04}) and the functoriality of the suspension operation we obtain
a natural factor map $\pi_{N,M} : (\tilde X_N, \tilde f_N) \tto (\tilde X_M, \tilde f_M)$ whenever, $M$ is a divisor of $N$.
That is, $\pi_{N,M}$ is the $M$-fold suspension of the map $\pi_{N/M}$.

In the language of \cite{AGW} the threads become \emph{spirals}.

Let $(X_n,f_n)$ be the disjoint union of $n!$ disjoint copies of the $n!$-fold suspension of $(\tilde X, \tilde f)$.
By decomposing the collection of copies in $f_{n+1}$ by sets of size $n+1$
and then using $\pi_{(n+1)!,n!}$ with a common range on each set of copies, we obtain a
an inverse sequence of maps $p_{n+1,n} : (X_{n+1},f_{n+1}) \tto (X_n,f_n)$.  Let $ (X_{\infty},f_{\infty}) $
be the inverse limit.  By choosing homeomorphisms to a common space $X$, we can regard $f_n \in H(X)$
for $1 \leq n \leq \infty$.

\begin{theo}\label{theo4.06} If $K_0 \subset H(X)$ is the union of the conjugacy classes of $f_n$ for $1 \leq n < \infty$
then $K_0$ is a conjugacy invariant collection of simple homeomorphisms which is dense in $C_s(X)$. The homeomorphism
$f_{\infty}$ is a conjugacy transitive point for $C_s(X)$. \end{theo}

\begin{proof} By Proposition \ref{prop4.01} every surjective relation on an element of $\I$ is a factor of a finite union of
dumbbells and every finite union of dumbbells is a factor of $f_n$ for $n$ sufficiently large. Hence,
$\Gamma(K_0) = \Gamma(C_s(X))$.  By Theorem \ref{theo3.11} (e) $K_0$ is dense in $K$.

Since every $f_n$ is a factor of $f_{\infty}$ it follows that
$\Gamma(C_s(X)) = \Gamma(K_0) \subset \Gamma(f_{\infty}) \subset \Gamma(C_s(X)).$
So $f_{\infty}$ is a conjugacy transitive homeomorphism.
\end{proof}\vspace{.5cm}

Let $\tilde X^+ = \{ e_- \} \cup \{(i,c) \in \tilde X : i \leq 0 \}$. Choose $c_0 \in C$ and let
$e_0 = (0,c_0)$ and define $\tilde f^+ \in C_s(\tilde X^+)$ by
\begin{equation}\label{4.05}
\tilde f^+(x) \quad = \quad \begin{cases} \tilde f(x) \quad \mbox{for}  \  x = e_-, \ \mbox{and} \ x = (i,c), i < 0, \\
e_0 = (0,c_0)  \quad \mbox{for} \ x = (0,c). \end{cases}
\end{equation}

Thus, $\tilde f^+$ is a surjective map and that every point except the fixed point $e_-$
is eventually mapped to the fixed point $e_0$.

Let $(X^+_n,f^+_n)$ be the disjoint union of $n!$ disjoint copies of the $n!$-fold suspension of $(\tilde X^+, \tilde f^+)$.
Again we can choose homeomorphisms to a common space $X$ and regard $f_n \in C_s(X)$ for $1 \leq n < \infty$.

\begin{theo}\label{theo4.07} If $K_0 \subset C_s(X)$ is the union of the conjugacy classes of $f^+_n$ for $1 \leq n < \infty$
then $K_0$ is a conjugacy invariant collection of very simple mappings which is dense in $C_s(X)$. \end{theo}

\begin{proof}  Just as in Theorem \ref{theo4.06}.
\end{proof}\vspace{.5cm}

Recall that for  $f \in C_s(X)$, $ Per(f) = \{ n : |f^n| \not= \emptyset \}$. So, for example, $f$ has a fixed point iff
$1 \in Per(f)$, in which case $Per(f) = \N$.
By Proposition \ref{prop1.05} the set
$$C_s(X;Per \supset Q) \quad = \quad \{ \ f \in C_s(X) \ : \ Q \subset Per(f) \ \}$$
is a closed, conjugacy invariant subset of $C_s(X)$ for any $Q \subset \N$. Recall
that we write $C_s(X;1)$ for $C_s(X;Per \supset \{ 1 \})$.
So $C_s(X;1)$ is the closed set of maps which admit a fixed point.

\begin{theo}\label{theo4.08} For any $Q \subset \N$, the set $C_s(X;Per \supset Q)$ is a conjugacy transitive subset
of $C_s(X)$. With  $1 \in Q$ the set  $C_s(X; 1)$  has a conjugacy
transitive element of residual type. \end{theo}

\begin{proof} Recall that if $R_2$ is a factor of $R_1$ then $Per(R_1) \subset Per(R_2)$. Hence, if
$\phi \in \Gamma(C_s(X; Per \supset Q))$ then $Q \subset Per(\phi)$. Conversely, if $Q \subset Per(\phi)$ then
Theorem \ref{theo3.02} implies that $\Gamma^{-1}(\phi) \cap C_s(X; Per \supset Q)$ is dense in $C_s(X; Per \supset Q)$.
It follows that $\Gamma(C_s(X; Per \supset Q)) $ consists exactly of the surjective relations $\phi$ on members of
$\I$ such that $Q \subset Per(\phi)$. Since $\Gamma^{-1}(\phi)$ is open, the Baire Category Theorem implies that
$C_s(X; Per \supset Q)$ is a conjugacy transitive set.

Now consider the fixed point case. It is clear from Proposition \ref{prop4.01} that every surjective relation $\phi$
with $|\phi| \not= \emptyset$ is a factor of a disjoint union of dumbbells together with a single extra vertex which is
related only to itself. Now we adjust the Shimomura sequence of Theorem \ref{theo4.04} to obtain
$P^*_{n+1,n} : (I^*_{n+1}, \Phi^*_{n+1}) \tto (I^*_{n}, \Phi^*_{n}) \}$ as follows:
At every level $n \geq 2$ we obtain $(J^*_n,\Phi^*_n)$
we adjoin a single new pair $\{1 \}, \{ (1,1) \}$. We alter $P_{n+1,n}$ so that $P^*_{n+1,n}$ it maps $1$ at level
$n+1$ and the dumbbell $I_{2,n+1}$ both to $1$ at level $n$. Above $1$ at level $m > n$ there is $1$ and $2^{m-n-1}$
dumbbells. So if $p$ is any map from level $k$ onto level $n$ we can allocate to the $k$ level dumbbells mapped to $1$ and
to $1$ at level $k$, all of the dumbbells at level $m$ which are mapped by $P_{m,n}$ down to $1$ at $n$.

Again we leave the details to the reader.

The result is an invertible Shimomura sequence satisfying the factoring property and with limit $f$ have a unique fixed
point over the $1$'s and with $\Gamma(f)$ containing arbitrarily large unions of large dumbbells.  It follows that
$f$ is a conjugacy transitive point for $C_s(X;1)$ which is of residual type.
\end{proof}\vspace{.5cm}

{\bfseries Example 2 -  $\mathbf{CR(X)}$ and $ \mathbf{ CT(X)}$:}
\vspace{.5cm}

Recall that  $CR(X)$ and $CT(X)$ are the closed, conjugacy invariant subsets of chain
recurrent elements of $C_s(X)$ and of chain transitive elements of $C_s(X)$, respectively.
By Proposition \ref{prop4.01}
$\Gamma(CR(X))$ is generated by the finite disjoint unions of loops and $\Gamma(CT(X))$ is generated by the loops.

\begin{theo}\label{theo4.09} (a) $CR(X)$ is a conjugacy transitive
subset of $C_s(X)$ with a transitive element of residual type.

(b) For $g \in C_s(X)$ chain recurrent, let $g_n \in C_s(X)$ be
isomorphic to $n$ disjoint copies of the $n!$-fold suspension of $g$.
If $K_0$ is the union of the conjugacy classes of $g_1, g_2, ... $, then $K_0$  is a
dense subset of $CR(X)$.

(c) If $K_1$ is the set of $f \in H(X)$ with a dense  set of periodic points, $K_1$ is
a dense subset of $CR(X)$.

(d) For any $Q \subset \N$, the set $C_s(X;Per \supset Q)\cap CR(X)$
is a conjugacy transitive subset of $C_s(X)$.The set
$ CR(X; 1)$  of chain recurrent maps
admitting a fixed point  has a conjugacy transitive point of
residual type.
\end{theo}

\begin{proof}(a) Theorem \ref{theo3.02} implies that for $\phi$ any recurrent relation
on an element of $\I$, the set $\Gamma^{-1}(\phi) \cap CR(X)$ is open and dense.  As there are countably many such relations
$\phi$, it follows from the Baire Category Theorem that $CR(X)$ is a transitive set.

Let $\{ k_n \}$ be an increasing sequence with $k_1 = 1$ and such that $k_n | k_{n+1}$.
For $n \geq 2$, let $\{ (I_n, \phi_{i,n} : i = 1,\dots,2^{n-1} \}$
be disjoint $k_n$-loops with $e_{i,n}$ a chosen point of $I_n$, and let $(J_n, \Phi_n)$ be their disjoint union.
Let $(J_1,\Phi_1) = ( \{ 1 \}, \{ (1,1) \})$ with $e_{11} = 1$. Define
$P_{n+1,n} : (J_{n+1}, \Phi_{n+1}) \to (J_n, \Phi_n)$ be the map uniquely defined by
$$ P_{n+1,n}(e_{i,n+1}) = e_{i,n} = P_{n+1,n}(e_{i + 2^{n-1},n+1}) \quad \mbox{ for}  \quad i = 1, \dots ,2^{n-1}.$$
Using the now familiar allocation argument and Lemma \ref{lem4.03}(a), it is easy to
check that this defines an invertible Shimomura sequence which satisfies the factoring
property.  The inverse limit is clearly chain recurrent. If we choose $k_n = n!$ then
every finite disjoint union of loops is a factor of $\Phi_n$ for $n$ sufficiently large, and so
the resulting inverse limit is a transitive element of $CR(X)$ of residual type.

(b) If $g$ is chain recurrent then every $g_n$ is chain recurrent and any finite disjoint
union of loops is a factor of $g_n$ for $n$ sufficiently large.
By Theorem \ref{theo3.11} (e) $K_0$ is dense in $CR(X)$.

(c) In (b) use $g \in H(X)$ with dense periodic points, e. g. the shift homeomorphism on $\{ 0,1 \}^{\Z}$.
Then every element of $K_0$ has dense periodic points, i.e. $K_0 \subset K_1$.

(d) Again Theorem \ref{theo3.02} implies that $C_s(X;Per \supset Q)\cap CR(X)$
is a conjugacy transitive subset of $C_s(X)$.

In the fixed point case we proceed exactly as in Theorem \ref{theo4.08}. We define
$P^*_{n+1,n} : (I^*_{n+1}, \Phi^*_{n+1}) \tto (I^*_{n}, \Phi^*_{n}) \}$ by adjoining
$(\{ 1 \}, \{ (1,1) \})$ at every level $n$ with $n \geq 2$. The we alter the map $P_{n+1,n}$ so that
$$ P^*_{n+1,n} (1) = 1 = P^*_{n+1,n}(e_{1,n+1}),$$
but is otherwise unchanged. Again it is easy to check that the sequence satisfies the factoring property.
It is clear that  $\Gamma(CR(X; 1))$ is generated by finite
unions of loops together with a single disjoint fixed point.
Hence, when $k_n = n!$ the limit is a conjugacy transitive element of $CR(X ; 1)$.
\end{proof}\vspace{.5cm}

Call $\{ k_n \}$ a \emph{divisibility sequence}\index{divisibility sequence} when it is an increasing
sequence of positive integers such that $k_1 = 1, k_n | k_{n+1}$ for all $n \in \N$.,
For such a sequence the sequence of group homomorphisms
$$\{ p_{n+1,n} : (\Z/k_{n+1}\Z,\phi_{n+1})  \tto (\Z/k_n \Z, \phi_n) \}$$ defines
an inverse sequence of loops where each $\phi_n$ is translation by $1$. Let $P^* = \{ k \in \N : k | k_n $ for
sufficiently large $ n \}$. The inverse limit $(X,f)$ is a minimal system which maps onto a $k$-loop iff $k \in P^*$.
It is called the \emph{adding machine}\index{adding machine} or \emph{odometer} associated with the sequence $\{ k_n \}$ although it really
depends only on the set $P^*$.  When $P^* = \N $, e.g. when $k_n = n!$, then the system is called the
\emph{universal adding machine}\index{adding machine!universal}.  It factors onto every loop.

It is easy to see that the homeomorphism constructed in part (a) above is the product of the identity on a Cantor set with
the adding machine  associated with the sequence $\{ k_n \}$. In particular, as noted by Shimomura, \cite{S3},
the product of the identity on a Cantor set
with the universal adding machine is a transitive element of residual type for $CR(X)$.

\begin{theo}\label{theo4.10} (a) If $\{ k_n \}$ is a divisibility sequence then
$$\{ \ p_{n+1,n} : (\Z/k_{n+1}\Z,\phi_{n+1})  \tto (\Z/k_n \Z, \phi_n) \ \}$$
is an invertible Shimomura sequence satisfying the factoring property and so the associated adding machine
is of residual type.

(b)  $CT(X)$ is a conjugacy transitive
subset of $C_s(X)$ with a the universal adding machine transitive element of residual type.

(c) For $g \in C_s(X)$ chain transitive, let $g_n \in C_s(X)$ be
isomorphic to the $n!$-fold suspension of $g$.
If $K_0$ is the union of the conjugacy classes of $g_1, g_2, ... $, then $K_0$  is a
dense subset of $CT(X)$.
\end{theo}

\begin{proof}  It is easy to check the factoring property for the inverse sequence associated
with a divisibility sequence.  For the universal adding machine every loop is a factor and so it is
a transitive element for $CT(X)$. This proves (a) and (b). Finally, (c) is proved just like (b) of
Theorem \ref{theo4.09}.
\end{proof}\vspace{.5cm}

Part (b) was first shown by Hochman in \cite{H}.

\begin{cor}\label{cor4.11} If $f \in C_s(X)$ is a conjugacy transitive element of $C_s(X), CR(X)$ or $CT(X)$ then
$f$ admits no periodic points, i. e. $|f^n| = \emptyset$ for all $n \in \N$. \end{cor}

\begin{proof} If $n \in Per(f)$ is a closed, conjugacy invariant condition. So if it is true for $f$, then
it is true for every element of the closure of the conjugacy class of $f$.  In particular, since the adding machines
have no periodic points, they cannot be the in the closure of such a conjugacy class.
\end{proof}\vspace{.5cm}

{\bfseries Example 3 - One Point Compactification of Incomparable Adding Machines:}
\vspace{.5cm}

It can happen that an element $f \in H(X)$ of residual type is the inverse limit of a sequence which does not satisfy the
factoring condition. This shows that the factoring property is a property of the sequence itself and not just of the
inverse limit.

Let $ p_i$ be the $i^{th}$ prime number in $\N$  counted in increasing order.
Observe that for $n, k \geq 1$  a $p_i^n$ loop cannot map to a $p_j^k$ if $i \not= j$.Fix a positive number $K$.
Let $(J_K,1_K) = (\{ 1, \dots, K \}, \{ (1,1), \dots, (K,K) \})$.

Now let $(I_1, \phi_1) = (J_1,1_1)$.  For $n > 1$ let $(I_n, \phi_n)$ be the disjoint union of
$(J_{K},1_{K})$ and a $p_i^n$ loop for $i = 1, \dots, Kn$. The map $p_{n+1,n} :(I_{n+1},\phi_{n+1}) \to (I_n,\phi_n)$
maps as follows:
\begin{itemize}
\item $(J_{K},1_{K})$ is mapped to the trivial loop $(J_1, 1_1) \})$ in $(J_K, 1_K)$.
\item The $p_i^{n+1}$ loop is mapped to $(\{ j \}, \{ (j, j) \})$ for $j = 1, \dots ,K$ and $i = Kn + j$.
\item The $p_i^{n+1}$ loop is mapped to the $p_i^n$ loop for $i= 1, \dots, Kn$.
\end{itemize}

For every $K$ the inverse sequence is an invertible Shimomura sequence and the inverse limits are all the same.  The common
limit  is the one-point compactification of the disjoint union of
$\{ p_i^n \}$ adding machines.

For $K = 1$ it is easy to check that the inverse sequence satisfies the factoring property.  Suppose $k > n > 1$ and $p$ maps
$(I_k, \phi_k)$ to $(I_n, \phi_n)$. Then $p$ must map
the $p_i^k$ loop to the $p_i^n$ loop for $i = 1, \dots, Kn = n$ and the remaining loops and the point $\{ 1 \}$ must map to
$(J_1, 1_1)$.  From this, the factoring is easy to obtain. It follows that the inverse limit is of residual type.

On the other hand, for any $K > 1$ the sequence does not satisfy the factoring property.  The map $p$ can only factor
when for $i = Kn + (k - n)K + j$ with $j = 1, \dots, K$, the $p_i^k$ loop is mapped to $ (\{ j \}, \{ (j,j) \})$.
But if $p$ does this then we can compose with a nontrivial permutation of $\{ 1, \dots, K \}$ to obtain a map
which does not factor.   Similarly, no subsequence satisfies the factoring property.

%
%
\newpage

{\bfseries Example 4 -  $\mathbf{CM(X)}$ and $\mathbf{CM(X;1)}$:}
\vspace{.5cm}

In general, a continuous surjective map on a compact metric space is chain mixing iff it is chain transitive and
in addition it has no nontrivial loop as a factor.  See \cite{A1} Chapter 8, Exercise 22, as well as
\cite{RW}.  In particular, if $f$ is chain transitive and has a fixed point then it is chain mixing. Thus,
$CM(X;1) = CM(X) \cap C_s(X;1) = CT(X) \cap C_s(X;1)$ is the set of chain mixing maps which admit a fixed point.

We recall a standard numerical lemma.

\begin{lem}\label{lem4.12} Given a pair of relatively prime, positive integers, every sufficiently large positive integer is a
positive mixture of these two.  In detail, for $M, N \in \N$ if $1 = xM - yN$ with $x,y \geq 0$ and $K \geq M + (yN + 1)N$ then
there exist $a, b \in \N$ such that $K = aM + bN$. \end{lem}

\begin{proof} For an integer $j \geq 0$ there are unique integers $j_1, j_2$ such that $j_1 \geq 0$, $N > j_2 \geq 0$ and
$j = j_1 N + j_2$. Then $(M + (yN + 1)N) + j = (1 + j_2x)M + (1 + j_1 + (N - j_2)y)N$.
\end{proof}\vspace{.5cm}

Recall that an $N-M$ \emph{wedge} is a system isomorphic to a $N-0-M$ dumbbell.  We extend Proposition \ref{prop4.01}.

\begin{prop}\label{prop4.13} (a) If $(I,\phi)$ is a mixing system with $I \in \I$ then there exists $K \in \N$ such that
it is the factor of any $N-M$ wedge with $N, M \geq K$.

(b) For an $N-M$ wedge with $M, N$ relatively prime there exists $K$ such that it is a factor of any loop of length $L \geq K$.

(c) If $(I,\phi)$ is an $N-M$ wedge with $M, N$ relatively prime then $\phi$ is mixing.
\end{prop}

\begin{proof}(a) There exists $K_1$ so that $\phi^N = I \times I$ for any $N \geq K_1$. There exists a loop of length $K_2$
which maps onto $\phi$.  It is then easy to see that if $M, N \geq K_1 + K_2$ and $i_0 \in I$ then there are loops of length
$M $ and $N$ each of which maps onto $\phi$ via maps which send an arbitrary point in each to $i_0$.  We can put these together
to map the $N-M$ wedge onto $\phi$.

(b) If $L = aM + bN$ with $a, b \in \N$ then any loop of length $L$ maps onto an $N-M$ wedge. So the result follows from Lemma \ref{lem4.12}.

(c) Choose $K$ as in (b). Let $L \geq K + M + N$.  If $i, j \in I$ then there is a path of length $q \leq M + N$ from $i$ to $j$ and there
is a loop of length $L - q \geq K$ from $i$ to $i$. Combine to get a path of length $L$ from $i$ to $j$.
\end{proof}\vspace{.5cm}

\begin{cor}\label{cor4.14} (a) If $\{ (N_n,M_n) \}$ is a sequence of relatively prime pairs of positive integers
with $ N_n \to \infty,  M_n \to \infty$ as $n \to \infty$ then the $N_n-M_n$ wedges generate $\Gamma(CM(X))$.

(b) If $K$ is a positive integer and $\{ M_n \}$ is a sequence of positive integers relatively prime to $K$
and $M_n \to \infty$ as $n \to \infty$ then the $K-M_n$ wedges generate $\Gamma(CM(X) \cap C_s(X; Per \supset \{ K \}))$.\end{cor}

\begin{proof}  (a) From Proposition \ref{prop4.13} (c) it follows that if $N, M$ are relatively prime then an $N-M$ wedge is
mixing.  From Theorem \ref{theo3.01} and Proposition \ref{prop4.13}(a) it follows that $\Gamma(CM(X))$ is generated
by the sequence of $N_n-M_n$ wedges.

(b) Assume $(I,\phi)$ is mixing with $|\phi^K| \not= \emptyset$.
If $ i \in |\phi^K|$ then there is a $K$ loop which maps into $\phi$ which begins and ends at $i$.  If $\phi$ is mixing
then it is a factor of $L$ loops for sufficiently large $L$. Hence, it is a factor of a $K-M_n$ wedge if $n$ is sufficiently
large. Again Theorem \ref{theo3.01} implies that any $K-M$ wedge with $K, M$ relatively prime represented by some
mixing homeomorphism $f$ with $K \in Per(f)$.
\end{proof}\vspace{.5cm}

Recall that $f \in C_s(X)$ is weak mixing when $f \times f$ is topologically transitive on $X \times X$. A factor of a weak
mixing system is weak mixing.  On the other hand, a nontrivial loop is not a weak mixing system.  It follows that a weak
mixing system does not have any nontrivial loop factors and so is chain transitive.

\begin{theo}\label{theo4.15} $CM(X)$ is a conjugacy transitive subset of $C_s(X)$. There exists $f \in H(X)$
with $(X,f)$ minimal, which is a conjugacy transitive element of $CM(X)$. If $f$ is a conjugacy transitive element for
$CM(X)$ then $f$ does not admit any periodic points.

The set of weak mixing, minimal homeomorphisms which are conjugacy transitive form a dense $G_{\dl}$ subset of
$CM(X)$.\end{theo}

\begin{proof}  Again it follows from Theorem \ref{theo3.02} that $CM(X)$ is conjugacy transitive. In detail for each
$\phi$ mixing, the theorem implies that $\Gamma^{-1}(\phi) \cap CM(X)$ contains a dense set of topologically mixing homeomorphisms.
Hence, the $G_{\dl}$ set $\Gamma^{-1}(\phi) \cap WM(X)$ is dense in $CM(X)$. Intersecting over the countable set of mixing $\phi$
we obtain a $G_{\dl}$ set, dense in $CM(X)$, each member of which is a transitive element of $CM(X)$.

We will construct a
conjugacy transitive element of $CM(X)$ which is minimal and so does not admit periodic points.
Just as in Corollary \ref{cor4.11} it
follows that no transitive element admits periodic points.

Let $N_1 = 1 = M_1$ and let $A = \begin{pmatrix}a & b \\ c & d\end{pmatrix}$ be a matrix of positive integers with
determinant $\pm 1$.  This implies that $A^{-1}$ is an integer matrix. Inductively, let
\begin{equation}\label{4.06}
\begin{pmatrix} M_{n+1} \\ N_{n+1} \end{pmatrix} \quad = \quad
\begin{pmatrix}a & b \\ c & d\end{pmatrix} \cdot \begin{pmatrix} M_{n} \\ N_{n} \end{pmatrix}.
\end{equation}

Clearly, $min(M_{n+1}, N_{n+1}) \ \geq \ M_n + N_n \ > \ max(M_n,N_n)$ and so $M_n, N_n \to \infty$.
Since $A^{-1}$ is an integer matrix it follows that if $d$ is a common divisor of $M_{n+1}$ and $N_{n+1}$
then it is a common divisor of $M_n$ and $N_n$. So, by induction, $(N_n,M_n)$ is a relatively prime pair for each $n$.

Let $p_{n+1,n}$ map each loop of the $N_{n+1}-M_{n+1}$ wedge onto the $N_n-M_n$ wedge with the common wedge point mapped
to the wedge point. For example, the $N_{n+1}$ loop maps $a$ times around the $N_n$ loop and $b$ times around the
 $M_n$ loop. By starting both loops around $N_n$ and ending both around $M_n$ it follows that, with the other choices
arbitrary, $p_{n+1,n}$ realizes a $\pm$ determined lift.  Hence, we obtain an invertible, pointed Shimomura sequence whose
inverse limit $(X,f)$ is a conjugacy transitive element of $CM(X)$.

We show that for every $x \in X$
$\{ f^k(x) : k \in \Z \}$ is dense in $X$.  Let $\al_n: (X,f) \to ([N_n + M_n -1],\phi_{N_n-0-M_n})$ be the projection
to the wedge at level $n$.  Either $\al_{n+1}(x)$ of $\al_{n+1}(f(x))$ is not at the wedge point.  Then $\al_{n+1}(f^k(x))$
for $k \in [-(N_{n+1} + M_{n+1} -1), N_{n+1} + M_{n+1} - 1]$ maps at least onto one of the loops. Then $\al_n =
p_{n+1,n} \circ \al_{n+1}$ maps onto both loops. It follows that $(X,f)$ is a minimal system.

By Proposition \ref{prop1.05} the set $MM(X)$ of minimal maps on $X$ is a  $G_{\dl}$ subset of $C_s(X)$. Such a map is chain
transitive and so is chain mixing exactly when it does not factor over a nontrivial loop. Hence, $H(X) \cap MM(X) \cap CM(X)$
is the $G_{\dl}$ set of such homeomorphisms.  Since it is conjugacy invariant and contains $f$ as constructed above, it
is dense in $CM(X)$. Since $WM(X)$ is a dense $G_{\dl}$ subset of $CM(X)$, the intersection $H(X) \cap MM(X) \cap WM(X)$
is dense in $CM(X)$. The conjugacy transitive elements $Trans(CM(X))$ is also a $G_{\dl}$, dense since it is nonempty.
Thus, $H(X) \cap MM(X) \cap WM(X) \cap Trans(CM(X))$ is a dense $G_{\dl}$ subset of $CM(X)$.
\end{proof}\vspace{.5cm}

Define $HM(X;1!)$ to be the set of homeomorphisms $f$ on $X$ which admit a unique fixed point and if $x \in X$ is not
fixed by $f$ then the $\pm$ orbit $\{ f^i(x) : i \in \Z \}$ is dense in $X$. Since $X$ is perfect, it follows that
such a homeomorphism is topologically transitive (see the Remark after Proposition \ref{prop1.05}). Since it has a
fixed point it is chain mixing.

\begin{theo}\label{theo4.16} There exists $f \in HM(X;1!)$ which is a transitive element for $CM(X;1)$. The set
$HM(X;1!) \cap WM(X) $ is a dense, $G_{\dl}$ subset of $CM(X;1)$.\end{theo}

\begin{proof} As in Theorem \ref{theo4.15}, the set $WM(X; 1)$ is a dense $G_{\dl}$ subset of $CM(X; 1)$ by Theorem
\ref{theo3.02}.

The $1-M$ wedge, or \emph{pointed loop}\index{pointed loop} of length $M$, is given by $([1,M],\phi_{1-M})$ with
$$\phi_{1-M} = \{ (i,i+1) : 1 \leq i < M \} \ \cup \{ (1,1),(M,1) \}.$$

Let  $\{ N_n \}$ be a sequence in $\N$ with $N_1 = 1$ and $N_{n+1} \geq 2N_n + 2$ for all $n \in \N$.
We define $p_{n+1, n} : ([1,N_{n+1}],\phi_{1-N_{n+1}}) \tto ([1,N_n],\phi_{1-N_n})$ so that
\begin{itemize}
\item $p_{n+1, n}(i) = 1 $ for $ i = 1, 2, N_{n+1} - 1, N_{n+1}$.
\item $p_{n+1, n}^{-1}(j)$ has at least two elements for all $j \in [1, N_n]$.
\end{itemize}
That is, the $N_{n+1}$ loop at level $n+1$ is wrapped by $p_{n+1, n}$ at least twice around the $N_n$ loop at level $n$.
It follows from these two conditions that each $p_{n+1, n}$ realizes a $\pm$ directional lift.  Hence,
 $\{p_{n+1, n} : ([1,N_{n+1}],\phi_{1-N_{n+1}}) \tto ([1,N_n],\phi_{1-N_n})  \}$
is an invertible, pointed Shimomura sequence. Let $(X,f)$ be the limit with
$\al_n : (X, f) \tto ([1,N_n],\phi_{1-N_n})$ the projection to the $n^{th}$ coordinate.

Let $e \in X$ be the point with $\al_n(e) = 1$ for all $n$. Clearly, $e$ is a fixed point
for $f$. If $x \not= e$ then $\al_n(x) \not= 1$ for some $n$ and so
$\al_m(x) \not= 1$ for all $m \geq n$.
It follows that $\al_m$ maps $\{ f^i(x) : |i| \leq N_m$ onto $[1,N_m]$ for all $m \geq n$.
This implies that the $\pm$ orbit of $x$ is dense in $X$.
That is, $f \in H(X;1!) \subset CM(X;1)$.

Since $\phi_{1-N_n} \in \Gamma(f)$ for all $n$, it follows from Corollary \ref{cor4.14} (b)
that $\Gamma(f) = \Gamma(CM(X;1))$ and so
$f$ is a conjugacy transitive element of $CM(X;1)$.

By Proposition \ref{prop1.05} the set $H(X;1!)$ is a $G_{\dl}$ subset of $CM(X;1)$.
Since it contains $f$ and is conjugacy invariant it is dense
in $CM(X;1)$. Hence, $H(X; 1) \cap WM(X) $ is dense in $CM(X;1)$
\end{proof}\vspace{.5cm}

We did not bother considering $Trans(CM(X;1))$ because, as we will now see, $CM(X;1)$ is conjugacy minimal,
i.e. every element of $CM(X;1)$ is conjugacy transitive for $CM(X;1)$.

We call a map $f$ \emph{periodic} when for some $n \in \N$, $|f^n| = X$, or, equivalently, $f^n = 1_X$.

\begin{theo}\label{theo4.17}(Shimomura)  If $f \in C_s(X)$ is not periodic then the closure of
the conjugacy class of $f$ contains $CM(X;1)$.
\end{theo}

\begin{proof} By Theorem \ref{theo3.13} it suffices to show that
$\Gamma(CM(X;1)) \subset \Gamma(f)$. By Corollary \ref{cor2.08} we can replace
$f$, if necessary, by its natural lift $\tilde f$ to a homeomorphism. Notice that if $f^n(x) \not= x$ and
 $\tilde x$ is a lift of $x$ then $\tilde f^n(\tilde x) \not= \tilde x$. In particular, if $f$ is not periodic then
 $\tilde f$ is not.

 So we assume $f \in H(X)$ is not periodic.  We show that for every $M \in \N$, there is a $1-M$ wedge in $\Gamma(f)$.
By Corollary \ref{cor4.14} (b) this will imply $\Gamma(CM(X; 1)) \subset \Gamma(f)$.

Since $f$ is not
periodic, there exists a point $x \in X$ such that
\\$x, f(x),...,f^{M+1}(x)$ are distinct points and so there is a clopen set $U$ with
$x \in U$ such that $\{ f^i(U) : i = 0,...,M + 1 \}$ are pairwise
disjoint. Let $\al : X \to I = \{1,...,M \}$ with $\al^{-1}(i) =
f^{i-1}(U)$ for $i = 2,...,M$ and $\al^{-1}(1) = X \setminus
\bigcup_{i=1}^{M-1} \ f^i(U)$.  Clearly, $f^{\al} $ is the $1-M$ wedge.
\end{proof}\vspace{.5cm}

{\bfseries Remark:} Using a more delicate proof, Shimomura shows, in \cite{S2}, that if $f$ is not periodic and
$g \in CM(X)$ then $g$ is in the orbit closure of $f$ iff $Per(f) \subset Per(g)$.
\vspace{.5cm}

\begin{cor}\label{cor4.18} In $C_s(X)$ the sets $\{ 1_X \}$ and $CM(X;1)$ are the
only conjugacy minimal subsets.\end{cor}

\begin{proof}  If $f \in C_s(X)$ is chain transitive then it is not periodic. In particular, the conjugacy class
of every element of $CM(X;1)$ is dense in $CM(X;1)$. That is, $CM(X;1)$ is a conjugacy minimal set.

The identity map commutes with every homeomorphism on $X$ and so $\{ 1_X \}$ is a fixed point for the $H(X)$ action.
It is therefore a conjugacy minimal set.

Clearly, $\Gamma(1_X) \ = \ \{ (I,1_I) : I \in \I \}$.  It follows that $1_X$ is in the orbit closure of $f$ iff
for every positive integer $N$ there exists a decomposition of $X$ of cardinality $N$ with each member an $f$ invariant set.

Now suppose that that $f$ is periodic with $f^n = 1_X$ for some $n \in \N$. Let
$d_f(x,y) \ = \ max_{i=0}^{n-1} \ d(f^i(x),f^i(y))$. This replaces $d$ with the topologically equivalent ultrametric $d_f$
and the latter is $f$ invariant, i.e. $f$ is an isometry.  Choose $\ep > 0$ small enough that the decomposition
$\{ V_{\ep}(x) :  x \in X \}$ contains at least $nN$ elements. Then, because $f$ is an isometry,
$\{ \bigcup_{i = 0}^{n-1} V_{\ep}(f^i(x)) : x \in X \}$ is a decomposition containing at least $N$ elements each of which is
$f$ invariant.  Hence, if $f$ is periodic, then $1_X$ is in the closure of its conjugacy class.

Thus, every closed, conjugacy invariant set contains either $\{ 1_X \}$ or $CM(X;1)$.
It follows that these are the only two conjugacy minimal sets.
\end{proof}\vspace{.5cm}

Notice that in Theorems \ref{theo4.15} and \ref{theo4.16} We did not describe any
residual transitive elements.  I conjecture that they
do not exist. In the construction of the former result there is such a wide range
of choices that it is hard to imagine  a construction can yield
a Shimomura sequence which satisfies the factoring property. The attempt in the latter case leads to an interesting semigroup.

Let $w = w(e,L)$ denote a finite word in the letters $e, L$ of length $N $, denoted $|w|$, so that $w \in \{ e,L \}^N$.
Let $k_w$ (and  $ K_w$) denote the number of letters $e $ (resp. the number of letters $L$)
in the word $w$. Define the affine function
$\ell_w$ by $\ell_w(x) = k_w + K_w x$. Let ${\mathcal S}$ denote the set of
words with $K_w > 0$. We define composition in ${\mathcal S}$ by
$w = w_1 * w_2$ with $w(e,L) = w_2(e,w_1(e,L))$. That is, substitute for every occurrence of
$L$ in $w_2$ the word $w_1$. If $m = \ell_w(n)$
then $w$ determines a unique map $p_{(w;m,n)} : ([1,m],\phi_{1-m}) \tto ([1,n],\phi_{1-n})$
which starts with a map of $1$ to $1$. As one moves along the
word, each $e$ indicates a map of a digit to $1$ and a move to the next digit in the domain,
and each $L$ indicates a map of $n$ digits in order to $[1,n]$
followed by a move to the next digit, except that after the final letter of $w$ is a stop
instead of a move to the next digit.  Conversely,
it is clear that if $p :  ([1,m],\phi_{1-m}) \tto ([1,n],\phi_{1-n})$ for any $m \geq n$,
then there is a unique word $w$ so that $p = p_{(w;m,n)}$ which implies
$m = \ell_w(n)$. It is easy to see that
\begin{equation}\label{4.07}
\begin{split}
\ell_{w_1 * w_2} \quad = \quad \ell_{w_2} \circ \ell_{w_1}, \hspace{3cm}\\
p_{(w_1 * w_2;m_2,n)} \quad = \quad p_{(w_1;m_1,n)} \circ p_{(w_2;m_2,m_1)}, \\
\mbox{with} \quad m_1 \ = \ \ell_{w_1}(n), \quad m_2 \ = \ \ell_{w_2}(m_1).
\end{split}
\end{equation}
For example, the word $L$ is the identity in ${\mathcal S}$ with
$\ell_L(x) = x$ and $p_{(L;n,n)}$ the identity on $([1,n],\phi_{1-n})$.

Let ${\mathcal S}'$ be the subsemigroup consisting of words $w$ which begin and end with $e$
and with $K_w \geq 2$. Let $\{ w_n \}$ be a sequence of
not necessarily distinct elements of ${\mathcal S}'$. Let $N_1 = 1$ and inductively define $N_{n+1} = \ell_{w_n}(N_n)$.
Let $p_{n+1,n} = p_{(w_n;N_{n+1},N_n)}$. This defines an invertible, pointed Shimomura sequence. The condition that the
words begin and end with $e$ is needed to get a $\pm$ directional lift.  If the limit is $(X,f) $ then $f \in H(X;1!)$.

The factoring property for the Shimomura sequence associated with $\{ w_n \}$ is equivalent
to following factorization property in the
semigroup:

For every $n \in \N$ and $w \in {\mathcal S}$ there exists $m \geq n$ and $\tilde w \in {\mathcal S}$ such that
\begin{equation}\label{4.08}
w * \tilde w \quad = \quad w_n * w_{n+1} * \dots * w_m.
\end{equation}
However, this is impossible for any sequence in ${\mathcal S}'$.  Consider the finite
list of positive integers which occur as the length of a run
of $L$'s in $w_a \in {\mathcal S}'$. Because $w_a$ begins and ends with $e$ it follows
that for any $w_b \in {\mathcal S}$ these are exactly the length of
runs in $w_a * w_b$.  In particular, given $n$, if we choose $w \in {\mathcal S}'$ so
that some length of runs occurs in $w$ but not in $w_{n}$ then
the factorization is not possible.

Nonetheless, the semigroup is of interest in studying $H(X;1!)$.

\begin{theo}\label{theo4.19} With $N_1 = 1$ construct the pointed, invertible Shimomura sequence associated with the sequence
$\{ w_n = w : n \in \N \}$ for $w \in {\mathcal S}'$.  Let $(X,f)$ denote the limit so that $f \in H(X; 1!)$.

(a) If $w = eLeLLe$ or $w = eLLLe$  then $(X,f)$ is topologically mixing.

(b) If $w = eLLe$ , the simplest word in ${\mathcal S}'$, or $w = eLeLe$ then $(X,f)$ is not weak mixing.

\end{theo}

\begin{proof} Let $w^n$ be the $n^{th}$ power of the element $w$ in the semigroup. Let $\al_n : X \to [1,N_n]$ be
the projection map from the inverse limit.

(a): For $w = eLeLLe$ we have $\ell_w(x) = 3 + 3x$ so that $N_{n+1} = 3 + 3N_n$
which implies $N_n = \frac{1}{2}[5 \cdot 3^{n - 1} - 3]$.
\vspace{.2cm}

{\bfseries Claim:}  In $w^n$ there occur between successive $L$'s runs of $e$'s of every length from $0$ up to
$2n - 1$.
\vspace{.5cm}

\emph{Proof of Claim}:  This is true by inspection for $n = 1$ since $w$ contains runs of length $1$ and $0$ between
successive $L$'s. Since $w^{n+1} = w * w^n$ we replace the $L$'s in $w^n$ by $w$'s to get $w^{n+1}$. This replaces each run
of length $k$ by a run of length $k + 2$. Thus, we obtain runs of $e$'s of length $2, \dots, 2n + 1 = 2(n+1) - 1$.
Within each $w$ occur runs of length $0$ and $1$.  So the Claim follows by induction.
\vspace{.2cm}

To show that $(X,f)$ is topologically mixing, it suffices to show that for every $n \geq 2$ and every $i \in [1,N_n]$
there exists $K \in \N$ so that the hitting time set $N(U_i,U_i)$ contains every integer greater than $K$ where
$U_i = \al^{-1}(i)$. Since $\al^{-1}(1)$ contains the fixed point the result is clear for $i = 1$.

Let $K = N_n$ and let $t = N_n + r$ with $0 \leq r \leq 2k - 1$ for some $k \in \N$.  Consider the map $p_{n+k,n}$
In $w^k$ there is a run of $e$'s of length $r$ between two successive $L$'s. If $j$ is the location of the
$i$ position in the first $L$ of the pair within $N_{n+k}$ then there exists $x$ with $\al_{n+k}(x) = j$.  Then $\al_n(x) = i$
and $\al_n(f^{t}(x)) = i$. In detail, iterating $f$ moves $x$ to position $1$ in $[1,N_{n}]$ where it then remains for
for $r$ steps and then it is moved back to position $i$ around the other end of the loop. Thus, $t \in N(U_i,U_i)$.
Hence, $(X,f)$ is topologically mixing.

For $w = eLLLe$, $N_{n+1} = 2 + 3N_n$ and so $N_n = 2 \cdot 3^{n-1} - 1$ with $N_n$ odd for
all $n$. For $n > 1$, as in the Claim above,
$w^n$ consists of
blocks $LLL$ separated by runs of $e$'s of even length from $2$ to $2(n-1)$.
If we have in $w^k$ two blocks of $LLL$ separated by a run of $e$'s of length $2r$,  then at level $n+k$
we choose $x$ at position $i$ in the third $L$ of the
first block of the pair.  We have $f^{t_1}(x)$ and $f^{t_2}(x)$ are at position $i$ in the
first and second $L$'s of the second block of the
pair when $t_1 = N_n + 2r$ and $t_2 = 2N_n + 2r$.  Since $N_n$ is odd we can choose $K = 2N_n + 2$ and proceed as above.
We leave the details to the reader.

(b): Since $\ell_w(x) = 2 + 2x$ we see that $N_{n+1} = 2 + 2N_{n}$ and so $N_n = 3 \cdot 2^{n-1} - 2$. In particular,
every $N_n$is even for $n \geq 2$.
 In $w^n$ there occur between successive $L$'s runs of $e$'s only of even length from
$0$ to $2n - 2$.

For any level $n \geq 2$ let $2 < i < N_n$. Assume that $t_1 \in N(U_i,U_i)$ and $t_2 \in N(U_i,U_{i-1})$.
Let $k > max(t_1, t_2)$. If $x_1, x_2 \in U_i$ with $f^{t_1}(x_1) \in U_i$ and $f^{t_2}(x_2) \in U_{i-1}$ then
at level $n+k$, neither $x_1$ nor $x_2$ lies in the portion of $[1,N_{n+k}]$ labeled by the last $L$ in $w^k$ because this
is followed by $k$ copies of $e$. This means that the pairs $x_1, f^{t_1}(x_1)$ and  $x_2, f^{t_2+1}(x_2)$
all lie in positions $i$ of $L$'s in $w^{k}$. But the length $N_n$ to which each $L$ maps via $p_{n+k,n}$ is even and
the number of $e$'s between them are even.  This implies that $t_1$ and $t_2+1$ are even and so $t_2$ is odd.
It follows that $N(U_i,U_i) \cap N(U_i,U_{i-1}) = \emptyset$ and so $(X,f)$ is not weak mixing.

For $w = eLeLe$, $N_{n+1} = 3 + 2N_n$ and so $N_n = 2^{n+1} - 3$ with $N_n$ odd for all $n$. The runs of $e$'s in $w^n$ between two
successive $L$'s  have odd length from $1$ to $2n-1$. Since each $N_n$ is odd,
the number of steps is even between the same location
not equal to $1$ in different $L$'s. Proceed as above. Details to the reader.
\end{proof} \vspace{.5cm}


%

{\bfseries Example 5 - Non Residual Factor }
\vspace{.5cm}

We conclude by observing that a factor of a homeomorphism of residual type need not be of residual type.

We showed in Proposition \ref{prop4.05} that $(\tilde X,\tilde f)$ which extends $T \times 1_C$ to the
 two-point compactification of $\Z \times C$ is of residual type.  If we let $(X, f)$ be the extension to
 the one-point compactification then $f$ is a chain transitive homeomorphism with a fixed point and so
 its conjugacy class is dense in the minimal set $CM(X;1)$. The system
 $(X,f)$ is obviously a factor of $(\tilde X, \tilde f)$. On the other hand,
 it is not topologically transitive and so its conjugacy class is disjoint from the $G_{\dl}$ set of topologically
 transitive maps with fixed points. By Theorem \ref{theo4.16} the latter set is dense in $CM(X; 1)$.
 By the Baire category theorem dense $G_{\dl}$ subsets meet. Hence, the dense conjugacy class of $f$ cannot be
 $G_{\dl}$. That is, $f$ is not of residual type.
\vspace{1cm}

\bibliographystyle{amsplain}

\begin{thebibliography}{10}


\bibitem{A1}
E. Akin, \emph{ The general topology of dynamical systems},
Graduate Studies in Mathematics, {\bf 1},  American Mathematical Society,
Providence, RI, 1993.

\bibitem{A2}
E. Akin, \emph{Good measures on Cantor space}, Trans. Amer. Math. Soc. {\bf 357} (2004), 2681--2722.

\bibitem{AHK}
E. Akin, M. Hurley and J. A. Kennedy, \emph{Dynamics of topologically generic homeomorphisms},
Memoirs  Amer. Math. Soc. {\bf 783},
American Mathematical Society, Providence, RI, 2003.

\bibitem{AGW}
E. Akin, E. Glasner and B. Weiss, \emph{Generically there is but one self-homeomorphism of the Cantor set},
Trans. Amer. Math. Soc. {\bf 360} (2008), 3613--3630.

\bibitem{AC}
E. Akin and J. D. Carleson, \emph{Conceptions of topological transitivity},Topology and its Applications
{\bf 159} (2012), 2815--2830.

\bibitem{BGT}
T. C. Batista, J. d.S. Gonschorowski and F. A. Tal, \emph{Density of the set of symbolic dynamics with all ergodic measures
supported on periodic orbits}, Fund. Math. {\bf 231}(2015), 81-92.

\bibitem{BD}
N. C. Bernardes Jr. and  U. B. Darji, \emph{Graph theoretic structure of maps of the Cantor space},
Advances in Mathematics, {\bf 231}  (2012), 1655--1680.

\bibitem{BDK}
S. Bezuglyi, A. H. Dooley and J. Kwiatkoski, \emph{Topologies on the group of homeomorphisms of the Cantor set},
(2004) ArXiv 0410507.

\bibitem{GW}
E. Glasner and B. Weiss, \emph{The topological Rohlin property and topological entropy},  Amer. J. Math.
{\bf 123} (2001), 1055--1070.

\bibitem{GW2}
E. Glasner and B. Weiss, \emph{Topological groups with Rohlin properties}, Colloq. Math. {\bf 110} (2008),  51-–80.

\bibitem{H}
M. Hochman, \emph{Genericity in topological dynamics}, Ergod. Th. \& Dynam. Sys. {\bf 28} (2008), 125--165.

\bibitem{KR}
A. S. Kechris and C. Rosendal, \emph{Turbulence, amalgamation and generic automorphisms of homogeneous structures},
 Proc. London Math. Soc. {\bf  94} (2007), 302--350.

\bibitem{K}
A. Kwiatkowska, \emph{The group of homeomorphisms of the Cantor set has ample generics},  Bull. London Math. Soc.
{\bf 44} (2012),  1132--1146.

\bibitem{RW}
D. Richeson and J. Wiseman, \emph{chain recurrence rates and topological entropy} Topology and its Applications
{\bf 156}  (2008), 251--261.

\bibitem{S1}
T. Shimomura, \emph{ A topological dynamical system on the Cantor set approximates its factors and its natural extension},
 Topology and its Applications {\bf 159} (2012), 3137--3142.

\bibitem{S2}
 T. Shimomura, \emph{ Aperiodic homeomorphisms approximate chain mixing endomorphisms on the Cantor set}, Tsukuba J.
Math. {\bf 36} (2013), 173--183.

\bibitem{S3}
T. Shimomura, \emph{Special homeomorphisms and approximation on Cantor systems}, Topology and its Applications
{\bf 161} (2014), 178--195.




\end{thebibliography}

\printindex

\end{document}